\documentclass[reqno,11pt]{amsart}
\usepackage{graphicx,epsfig}
\usepackage{epstopdf}
\usepackage{pdfpages}
\usepackage[labelfont=up]{subcaption}
\usepackage{amssymb}
\usepackage{empheq}
\usepackage{cases}
\usepackage{amsthm,amsmath}
\usepackage{algorithm}
\usepackage{algpseudocode}
\usepackage{caption,lipsum}
\usepackage{stmaryrd}
\usepackage{tabularx}
\usepackage{color}
\usepackage{empheq,float}
\DeclareGraphicsExtensions{.eps}
\usepackage{amssymb,amsmath,graphicx,amsfonts,euscript,mathrsfs}
\usepackage{bm}
\usepackage{hyperref}
\usepackage{enumerate}

\newtheorem{theorem}{Theorem}[section]
\newtheorem{lemma}[theorem]{Lemma}

\theoremstyle{definition}

\newtheorem{example}[theorem]{Example}

\newtheorem{assumption}[theorem]{Assumption}
\theoremstyle{remark}
\newtheorem{remark}{Remark}[section]

\numberwithin{equation}{section}

\newcommand{\rd}{{\mathrm{d}}}
\newcommand{\E}{{\mathbb E}}
\newcommand{\bDelta}{\boldsymbol{\Delta}}
\newcommand{\bN}{{\mathbb N}}
\newcommand{\bT}{{\mathbb T}}
\newcommand{\bR}{{\mathbb R}}
\newcommand{\bmu}{\boldsymbol{\mu}}

\newcommand{\bPhi}{\boldsymbol{\Phi}}
\newcommand{\bxi}{\boldsymbol{\xi}}
\newcommand{\by}{\boldsymbol{y}}
\newcommand{\bz}{\boldsymbol{z}}

\newcommand{\cB}{\mathcal{B}}
\newcommand{\cD}{\mathcal{D}}

\newcommand{\cI}{\mathcal{I}}
\newcommand{\cJ}{\mathcal{J}}
\newcommand{\cN}{\mathcal{N}}

\newcommand{\cW}{\mathcal{W}}

\newcommand{\re}{\mathrm{e}}
\newcommand{\ri}{\mathrm{i}}

\begin{document}
	\title[QMC-TS for Schr\"odinger equation with Gaussian random potential]{Quasi-Monte Carlo time-splitting methods for the Schr\"odinger equation with Gaussian random potential}
	
	\author[Z. Wu]{Zhizhang Wu}
	\address{\hspace*{-12pt}Z.~Wu: Department of Mathematics, The University of Hong Kong, Pokfulam Road, Hong Kong SAR, China. Now at Department of Mathematics, Hong Kong Baptist University, Kowloon Tong, Hong Kong SAR, China.}
	\email{wuzz25@hkbu.edu.hk, wzz14@tsinghua.org.cn}
	
	\author[Z. Zhang]{Zhiwen Zhang}
	\address{\hspace*{-12pt}Z.~Zhang: Department of Mathematics, The University of Hong Kong, Pokfulam Road, Hong Kong SAR, China. Materials Innovation Institute for Life Sciences and Energy (MILES), HKU-SIRI, Shenzhen, China.}
	\email{zhangzw@hku.hk}
	
	\author[X. Zhao]{Xiaofei Zhao}
	\address{\hspace*{-12pt}X.~Zhao: School of Mathematics and Statistics \& Computational Sciences Hubei Key Laboratory, Wuhan University, Wuhan, 430072, China.}
	\email{matzhxf@whu.edu.cn}
	\urladdr{http://jszy.whu.edu.cn/zhaoxiaofei/en/index.htm}
	
	\begin{abstract}\noindent
		In this paper, we study the Schr\"odinger equation with a Gaussian random potential (SE-GP) and develop an efficient numerical method to approximate the expectation of physical observables. The unboundedness of Gaussian random variables poses significant difficulties in both sampling and error analysis. Under time-splitting discretizations of SE-GP, we establish the regularity of the semi-discrete solution in the random space. Then, we introduce a non-standard weighted Sobolev space with properly chosen weight functions, and obtain a randomly shifted lattice-based quasi-Monte Carlo (QMC) quadrature rule for efficient sampling. This approach leads to a QMC time-splitting (QMC-TS) scheme for solving the SE-GP. We prove that the proposed QMC-TS method achieves a dimension-independent convergence rate that is almost linear with respect to the number of QMC samples. Numerical experiments illustrate the sharpness of the error estimate.
		\\ \\
		{\bf Keywords:} Schr\"odinger equation; Gaussian random potential; quasi-Monte Carlo (QMC) method; time splitting; error estimate; optimal rate. \\ \\
		{\bf AMS Subject Classification:} 65C30; 65D32; 65M15; 82B44.
	\end{abstract}
	
	\maketitle
	
	\section{Introduction}
	\label{sec: introduction}
	
	The Schr\"{o}dinger equation with a spatial random potential of the following form
	\begin{equation}\label{eq: original Schrodinger}
		\ri \partial_t\psi(t,\omega,x) = -\frac{1}{2} \partial_x^2 \psi(t,\omega,x) + V(\omega,x) \psi(t,\omega,x),
	\end{equation}
	plays an essential role for describing wave propagation in disordered media \cite{Conti,JMP,Lifshits},
	where $t$ is the time variable, $x$ is the space variable, $\omega$ is a random sample, $\psi=\psi(t,\omega,x)$ is the unknown complex-valued wavefunction, and $V$ is an external real-valued spatial random potential. It is also known as the continuous version of the original Anderson model \cite{anderson1958absence} for the localization phenomenon and is mathematically of great interest \cite{Debussche2,Debussche3,Dumaz,Klein,zhao2021numerical}. In practice, the spatial random potential $V$ could be uniformly distributed in a bounded interval \cite{Flach} or be Gaussian noise \cite{Conti,Gauss1}.  In this work, we will be interested in the case that $V(\omega, x) =v_0(x)+V_{\mathrm{r}}(\omega,x)$ with $v_0(x)$ a deterministic function and $V_{\mathrm{r}}(\omega,x)$ a zero-mean Gaussian random field, and we aim to provide an efficient numerical algorithm to solve (\ref{eq: original Schrodinger}) particularly addressing the sampling issue. Throughout this paper, we focus on the one-space-dimensional case for simplicity, i.e., $x \in \bR$; we refer the reader to Remark \ref{rem: high-space-dimensional cases} and Point (2) in Section \ref{subsec: discussion} for discussions on high-space-dimensional cases.
	
	Although the standard Monte Carlo (MC) method is handy as used in most of the simulation work, e.g., \cite{Conti,Flach,zhao2021numerical}, we get only a half-order convergence rate in the number of samples, and hence a large number of simulations are required for MC to accurately evaluate the statistical quantities. More efficient options for discretizing the random space include the quasi-Monte Carlo (QMC) methods \cite{caflisch1998monte,dick2022lattice,dick2013high,graham2015quasi,kuo2016application,kuo2011quasi,wang2003strong}, stochastic Galerkin methods \cite{cohen2010convergence,ghanem1991stochastic,hu2015stochastic,xiu2002wiener}, stochastic collocation methods \cite{babuvska2007stochastic,nobile2008sparse,tang2010convergence}, etc. Particularly for solving (\ref{eq: original Schrodinger}), \cite{wu2016bloch} and \cite{wu2020convergence} applied the stochastic Galerkin method and the stochastic collocation method, respectively, while due to the curse of dimensionality with respect to the number of samples \cite{nobile2008sparse,shu2017stochastic}, only cases with a one-dimensional random variable were considered. Then, a QMC approach for (\ref{eq: original Schrodinger}) under uniformly distributed $V$ has been considered in our previous work \cite{wu2024error}, where an almost linear convergence rate with dimension-independence is achieved.
	The analysis and the QMC quadrature rule proposed therein rely on the boundedness of $V$ and would fail for a Gaussian random field. Thus, we continue in this work the development of QMC towards the Schr\"odinger model but with a Gaussian random potential. We remark that QMC has also been applied to Schr\"{o}dinger equations in the deterministic case in \cite{suzuki2018rank,suzuki2019strang}, where QMC sample points and the corresponding anti-aliasing set serve as lattices for the Fourier pseudospectral method for spatial discretization.
	
	In view of the Karhunen--Lo\`{e}ve expansion \cite{ghanem1991stochastic,karhunen1947uber,loeve1948fonctions}, the Gaussian random potential in \eqref{eq: original Schrodinger} admits the following parametric representation:
	\begin{align} \label{eq: random-potential parametric}
		V(\omega,x) = V(\bxi(\omega), x) = v_0(x) + V_{\mathrm{r}}(\bxi(\omega),x), \text{ where } V_{\mathrm{r}}(\bxi(\omega),x) = \sum_{j=1}^{\infty} \lambda_j \xi_j(\omega) v_j(x),
	\end{align}
	with $\{ v_j(x) \}_{j = 1}^{\infty}$ the physical components, $\lambda_1 \ge \lambda_2 \ge \cdots > 0$ the corresponding strengths,  $\{ \xi_j(\omega) \}_{j = 1}^{\infty}$ the independent and identically distributed (i.i.d.) standard Gaussian random variables and $\bxi(\omega) = (\xi_1(\omega),\xi_2(\omega), \ldots)^{\top}\in \bR^{\bN}$ (see, e.g., \cite{charrier2012strong,graham2015quasi} for more details on the parametric representation of a Gaussian random field in the form of Karhunen--Lo\`{e}ve expansion). The law of $\bxi$ is defined on the product probability space $(\bR^{\bN}, \mathcal{B}(\bR^{\bN}), \bmu_{\infty})$, where $\mathcal{B}(\bR^{\bN})$ is the sigma-algebra generated by the cylinder sets, and $\bmu_{\infty}$ is the product Gaussian measure  \cite{bogachev1998gaussian}, i.e., $\bmu_{\infty} = \bigotimes_{j = 1}^{\infty} \cN(0, 1)$. Then, by the Doob--Dynkin lemma \cite{kallenberg2021foundations}, the solution $\psi$ of (\ref{eq: original Schrodinger}) can be represented by a function parameterized by $\bxi$. Hence, we can consider the initial value problem of \eqref{eq: original Schrodinger} in the parametric form as
	\begin{equation} \label{eq: Schrodinger}
		\left\{
		\begin{aligned}
			&\ri \partial_t \psi(t,\bxi,x) = -\frac{1}{2} \partial_x^2 \psi(t,\bxi,x) + V(\bxi,x) \psi(t,\bxi,x),
			\quad x\in \bT,\ \bxi \in U := \bR^{\bN},\ t>0,\\
			&\psi(t=0,\bxi,x) = \psi_{\mathrm{in}}(x),\quad x\in \bT,\ \bxi \in U,
		\end{aligned}
		\right.
	\end{equation}
	where $\psi_{\mathrm{in}}$ represents a prescribed (deterministic) initial wave and $\bT$ is the one-dimensional torus (periodic boundary).  The torus domain serves as a valid approximation of the whole space problem when the initial localized wave is yet to reach the boundary.
	
	Our aim is still for a dimension-independent first-order QMC approximation of (\ref{eq: Schrodinger}), since a random potential from reality can be rough, so the series in (\ref{eq: random-potential parametric}) decays very slowly. To achieve this goal,
	we choose to work under the framework of QMC with the randomly shifted lattice rule \cite{dick2004convergence,dick2022lattice,kuo2003component,kuo2010randomly,kuo2006randomly,nichols2014fast,nuyens2006fast,sloan2002constructing,waterhouse2006randomly}, which has been successfully developed for serval random PDEs. In the case of bounded random variables, we find its application to elliptic equations \cite{kuo2012quasi}, eigenvalue problems \cite{gilbert2019analysis}, optimal control \cite{guth2021quasi,guth2024parabolic}, Helmholtz equations \cite{ganesh2021quasi,graham2025quasi} and Schr\"{o}dinger equations \cite{wu2024error}. For the case of unbounded random variables, the relevant literature is quite limited: \cite{graham2015quasi} pioneered the work for elliptic equations with Gaussian random coefficients. In recent years, there have also been many studies discussing the application of QMC to PDEs with random inputs that belong to the Gevrey class, which is a much more general class that goes beyond the affine parameterization we consider in \eqref{eq: random-potential parametric}; see, e.g., \cite{chernov2024analyticSINUM,chernov2024analyticCMA,harbrecht2024gevrey}. In particular, \cite{chernov2024analyticSINUM,chernov2024analyticCMA} mainly considered elliptic PDEs, and \cite{harbrecht2024gevrey} proposed an abstract framework for analyzing the parametric regularity of a solution to a random PDE. However, the framework in \cite{harbrecht2024gevrey} cannot be applied directly to our problem since the Gaussian random variables we consider are unbounded and thus the Fr\'{e}chet derivatives of the residual equation with respect to the solution are not uniformly bounded for all data as required by \cite{harbrecht2024gevrey}; see \cite{harbrecht2024gevrey} for more details. Here, we still focus on the affine parameterization \eqref{eq: random-potential parametric} of $V$ for simplicity, and we would adopt a different approach from those in \cite{chernov2024analyticSINUM,chernov2024analyticCMA,harbrecht2024gevrey} for analyzing the parametric regularity of the solution. But we believe that the theory in \cite{harbrecht2024gevrey} can also be applied to our problem after appropriate modifications. In fact, due to its unboundedness, the Gaussian random variable in \eqref{eq: Schrodinger} poses much greater challenges to approximations and analysis than the uniformly distributed random variable case considered in \cite{wu2024error}. Although integration against unbounded random variables can be mapped into an integral over the unit cube using the inverse cumulative distribution function, the transformed integrand may be unbounded and may not even have square-integrable mixed first derivatives. Therefore, the QMC theory of standard weighted Sobolev spaces \cite{dick2013high,kuo2011quasi} cannot be applied.
	Instead, it is crucial to find the proper decaying weight functions that can counteract the growth of the mixed first derivatives of the integrand and meanwhile lead to the desired convergence rate \cite{kuo2010randomly,nichols2014fast}. In general, slower decay of weight functions (or equivalently slower growth of mixed first derivatives) results in faster convergence of QMC.
	
	To this end, we first adopt the time-splitting scheme which is one of the most popular classes of methods \cite{bao2012mathematical,jin2011mathematical,lubich2008splitting} for time discretization of Schr\"odinger models. The resulting subflows are self-adjoint, which yields a polynomially-growing bound on the mixed first derivatives of the semi-discrete solution with respect to $\bxi$. Such growth bounds lead to favorable weight functions. Then, by means of a non-standard weighted Sobolev space associated with the chosen weight functions, we derive the QMC quadrature rule that ends up as a class of QMC time-splitting schemes. Rigorous error estimates establish the desired dimension-independent and almost first-order convergence rate of QMC with the optimal temporal error bound. The theoretical results are validated by numerical experiments.
	
	The rest of the paper is organized as follows. In Section \ref{sec: main result}, we present the QMC time-splitting (QMC-TS) scheme and the main result on its convergence. The derivation of the main result (i.e., the convergence analysis) is given in detail in Section \ref{sec: convergence}. Then, we show numerical results to verify the convergence rates of our numerical method in Section \ref{sec: numerical example}. Finally, some concluding remarks are given in Section \ref{sec:Conclusion}.
	
	\paragraph{\textbf{Notation}} We will usually omit the variables $t, \bxi, x$ in the functions for notational brevity when there is no confusion caused. For any $1 \le q < \infty$ and any temporal-spatial function space $W$, we define the space $L_{\bmu_{\infty}}^q(U, W)$ equipped with the norm $\| \cdot \|_{L_{\bmu_{\infty}}^q(U, W)}$ such that for any $f(t, \bxi, x) \in L_{\bmu_{\infty}}^q(U, W)$, $f(\cdot, \bxi, \cdot) \in W$ for almost surely (a.s.) $\bxi \in U$ and $\| f \|_{L_{\bmu_{\infty}}^q(U, W)} := \left( \int_U \left( \| f(\bxi) \|_W \right)^q \rd \bmu_{\infty}(\bxi) \right)^{1/q} < \infty$.
	
	\section{Numerical method and main result}
	\label{sec: main result}
	
	We present the quasi-Monte Carlo time-splitting (QMC-TS) scheme to solve the Schr\"{o}dinger equation \eqref{eq: Schrodinger}.
	
	\subsection{Numerical method}
	\subsubsection{Dimension truncation}
	
	From the perspective of numerical computations, we will in practice work with the following truncated Schr\"{o}dinger equation:
	\begin{equation} \label{eq: Schrodinger trun}
		\left\{
		\begin{aligned}
			&\ri\partial_t\psi_m(t,\bxi_m,x) = -\frac{1}{2} \partial_x^2 \psi_m(t,\bxi_m,x) + V_m(\bxi_m,x) \psi_m(t,\bxi_m,x),
			\quad x\in \bT, \bxi_m \in \bR^m, t>0,\\
			&\psi_m(t=0,\bxi_m,x) = \psi_{\mathrm{in}}(x),\quad x\in \bT, \bxi_m \in \bR^m,
		\end{aligned}
		\right.
	\end{equation}
	where
	\begin{align} \label{eq: random-potential}
		V_m(\bxi_m, x) = v_0(x) + V_{\mathrm{r}, m}(\bxi_m, x), \text{ with } V_{\mathrm{r}, m}(\bxi_m, x) = \sum_{j=1}^{m} \lambda_j \xi_j(\omega) v_j(x),
	\end{align}
	and $\bxi_m = (\xi_1, \ldots, \xi_m)^{\top} \in \bR^m$. Note that any function of $\bxi_m \in \bR^m$ can also be seen as a function of $\bxi \in U$. The first step of QMC-TS is to approximate $\psi(t, \bxi, x)$ by $\psi_m(t, \bxi_m, x)$.
	
	\subsubsection{Time discretization}
	
	For each fixed $\bxi \in U$, \eqref{eq: Schrodinger trun} becomes a deterministic equation, and the second step of QMC-TS lies in the time discretization for \eqref{eq: Schrodinger trun} using the time-splitting scheme. It begins by splitting \eqref{eq: Schrodinger trun} into two subflows $\Psi^\mathrm{p}_{\rho}$ and $\Psi^\mathrm{k}_{\rho}$ as
	\begin{subequations} \label{eq: subflow}
		\begin{align}
			\Psi^\mathrm{p}_{\rho}: i\partial_t\psi_m &= V_m\psi_m, \quad t\in(0,\rho], \\
			\Psi^\mathrm{k}_{\rho}: i\partial_t\psi_m &= -\frac{1}{2} \partial_x^2 \psi_m,\quad t\in(0,\rho].
		\end{align}
	\end{subequations}
	Note that the above two equations can be integrated exactly in time since $V_m$ is real-valued. Let $\tau > 0$ be the time step size. In view of the Lie--Trotter product formula \cite{trotter1959product}, we employ the following first-order Lie--Trotter splitting method:
	\begin{align} \label{eq: Lie splitting}
		\psi_m^{n + 1} = \Psi^\mathrm{k}_{\tau} \circ \Psi^\mathrm{p}_{\tau}(\psi_m^{n}) = \re^{\ri \tau \partial_x^2 / 2} \re^{- \ri \tau V_m} \psi_m^n, \quad n = 0, 1, \ldots,
	\end{align}
	where $\psi_m^n$ is the approximation of $\psi_m(t_n)$, with $t_n = n \tau$ and $\psi_m^0 = \psi_{\mathrm{in}}$.
	The second step of QMC-TS is to approximate $\psi_m(t_n, \bxi_m, x)$ by $\psi_m^n(\bxi_m, x)$.
	
	\begin{remark} \label{rem: intermediate solution}
		Let $\psi_m^{n, *} = \Psi^\mathrm{p}_{\tau}(\psi_m^{n - 1}) = \re^{- \ri \tau V_m} \psi_m^{n - 1}$ for $n \ge 1$. Then, $\psi_m^{n} = \Psi^\mathrm{k}_{\tau}(\psi_m^{n, *}) = \re^{\ri \tau \partial_x^2 / 2} \psi_m^{n, *}$, and $\psi_m^{n, *} = g_n(\tau)$, where $g_n$ satisfies $g_n(0) = \psi_m^{n - 1}$ and
		\begin{align} \label{eq: intermediate solution}
			\ri \partial_t g_n = V_m g_n, \quad t \in (0, \tau].
		\end{align}
	\end{remark}
	
	\begin{remark}
		We can also use other high-order splitting schemes \cite{thalhammer2008high} of the general form
		\begin{align} \label{eq: high-order splitting}
			\psi_m^{n + 1} = \prod_{j = 1}^M \re^{\ri \alpha_j \tau \partial_x^2 / 2} \re^{- \ri \beta_j \tau V_m} \psi_m^{n}, \quad n = 0, 1, \ldots,
		\end{align}
		where $\alpha_j, \beta_j \in \bR$. We devote the analysis to the Lie--Trotter splitting scheme \eqref{eq: Lie splitting} for simplicity of presentation, and we will elaborate more on high-order splitting schemes in Section \ref{subsec: discussion}.
	\end{remark}
	
	\subsubsection{Quasi-Monte Carlo quadrature}

Note that $|\psi|^2$ represents the probability density function (or position density) in quantum mechanics. Many widely considered physical observables are linear functional of $|\psi|^2$. Therefore, in the following, we shall focus on an efficient scheme to approximate the expected value of $G(|\psi|^2)$, where $G$ is a linear functional.

We approximate $\E[G(|\psi(t_n)|^2)]$ by $\E[G(|\psi_m^n|^2)]$ with the above two steps. Then, we let $\phi(y) = \exp(- y^2 / 2) / \sqrt{2 \pi}$ be the density function of the standard univariate Gaussian distribution, $\Phi(y) = \int_{-\infty}^y \phi(\rho) \rd \rho$ be the cumulative distribution function, and $\Phi^{-1}$ be the inverse of $\Phi$. Moreover, we define the vector inverse Gaussian cumulative distribution function $\bPhi_m^{-1}$ such that $\bPhi_m^{-1}(\by)$ applies $\Phi^{-1}$ to $\by \in \bR^m$ component-wise. Then, by change of variables $\bxi_m = \bPhi_m^{-1}(\by)$, we have
	\begin{align} \label{eq: dimension truncation of expectation}
		\E[G(|\psi_{m}^n|^2)] = \int_{(0, 1)^m} F(\bPhi_m^{-1}(\by)) \rd \by,
	\end{align}
	where $F(\cdot) = G(|\psi_{m}^n(\cdot)|^2)$. The last step of QMC-TS is to approximate \eqref{eq: dimension truncation of expectation} by the QMC quadrature
	\begin{align} \label{eq: QMC for F}
		Q_{m, N}(F; \bDelta) = \frac{1}{N} \sum_{j = 1}^N F \left(  \bxi_m^{(j)} \right),
	\end{align}
	where $\{ \bxi_m^{(j)} \}_{j = 1}^N$ are the quadrature points with $N$ the number of samples. In particular, we adopt the randomly shifted (rank-1) lattice rule \cite{kuo2010randomly,nichols2014fast}, which generates the quadrature points as
	\begin{align}
		\bxi_m^{(j)} = \bPhi_m^{-1} \left( \mathrm{frac} \left( \frac{j\bz}{N} + \bDelta \right) \right), \quad j = 1, \ldots, N,
	\end{align}
	where $\bDelta \in [0, 1]^m$ is a random shift uniformly distributed over $[0, 1]^m$, $\mathrm{frac}(\by)$ takes the fractional part of $\by \in \bR^m$ component-wise, and $\bz \in \bN^m$ is known as the generating vector. A specific generating vector $\bz$ for the particular PDE problem can be constructed efficiently by the component-by-component (CBC) algorithm \cite{nichols2014fast}. We will elaborate more on the generating vector $\bz$ in Section \ref{subsec: QMC error}.
	
	Combining the above three steps, the QMC-TS method is summarized in Algorithm \ref{algo: QMC-TS}.
	
	\begin{algorithm}[t!]
		\caption{Quasi-Monte Carlo time-splitting method}
		\label{algo: QMC-TS}
		\begin{algorithmic}[1]
			\Statex \textbf{Input}: truncation dimension $m$, time step size $\tau$, number of samples $N$, number of random shifts $R$.
			\State Construct the generating vector $\bz$ by the CBC algorithm.
			\State Generate i.i.d. random shifts $\bDelta_1, \ldots, \bDelta_R$ from the uniform distribution on $[0, 1]^m$. For each $k = 1, \ldots, R$, obtain the sample set $\{ \bxi_m^{(k, j)} = \bPhi_m^{-1} \left( \mathrm{frac} \left( \frac{j\bz}{N} + \bDelta_k \right) \right): j = 1, \ldots, N \}$.
			\For {$k = 1:R$}
			\For {$j = 1:N$}
			\State Approximate the Schr\"{o}dinger equation \eqref{eq: Schrodinger trun} via Lie--Trotter splitting \eqref{eq: Lie splitting} or other high-order splitting schemes for each $\bxi_m^{(k, j)}$ and $n \in \bN$, and obtain $\psi_{m}^n(\bxi_m^{(k, j)})$.
			\EndFor
			\EndFor
			\State The approximation of the expectation of the physical observable $\E[G(|\psi(t_n)|^2)]$ is given by
			\begin{align}
				\overline{Q}_{m, N, R}(G(|\psi_{m}^n|^2)) = \frac{1}{R} \sum_{k = 1}^R Q_{m, N}(G(|\psi_m^n|^2); \bDelta_k) = \frac{1}{R N} \sum_{k = 1}^R \sum_{j = 1}^N G(| \psi_{m}^n(\bxi_m^{(k, j)}) |^2).
			\end{align}
			\Statex \textbf{Output}: $\overline{Q}_{m, N, R}(G(|\psi_{m}^n|^2))$.
		\end{algorithmic}
	\end{algorithm}
	
	\subsection{Main result}
	\label{subsec: main result}
	
	Let $s \ge 1$ be fixed. We make the following assumptions to guarantee the convergence of the QMC-TS method.
	
	\begin{assumption}\label{assp: general}
		Assume that $\psi_{\mathrm{in}}, v_0 \in H^s(\bT)$, $v_j \in H^s(\bT) \bigcap W^{1, \infty}(\bT)$ for $j \in \bN^{+}$, and the linear functional $ G \in (H^{1}(\bT))' $, which is the dual space of $H^{1}(\bT)$.
	\end{assumption}
	
	\begin{assumption} \label{assp: summability of b}
		Let $a_{j} = \lambda_j \| v_j \|_{H^s(\bT)}$ and $b_j = \lambda_j \| v_j \|_{W^{1, \infty}(\bT)}$ for $j \in \bN^{+}$. Assume that
		\begin{align*}
			\sum_{j = 1}^{\infty} a_j < \infty, \text{ and } \sum_{j = 1}^{\infty} b_{j}^p < \infty, \text{ for some } p \in (0, 1].
		\end{align*}
	\end{assumption}
	
	\begin{assumption} \label{assp: convergence of V}
		Assume that for some constants $C, \varepsilon, \chi>0$ independent of $m$
		\begin{align*}
			\| V_m - V \|_{L_{\bmu_{\infty}}^{2 + \varepsilon}(U, H^1(\bT))}\leq Cm^{-\chi}.
		\end{align*}
	\end{assumption}
	
	\begin{remark}
		If $v_j \in H^s(\bT)$ with $s > \frac{3}{2}$, we immediately have $v_j \in W^{1, \infty}(\bT)$ by Sobolev embedding \cite[Chapter 4]{adams2021sobolev}.
	\end{remark}
	
	Now we are ready to present the main result of this paper.
	
	\begin{theorem}
		\label{thm: main}
		Let Assumptions \ref{assp: general}--\ref{assp: convergence of V} hold with $s \ge 3$, and we additionally assume \eqref{eq: additional assumption} if Assumption \ref{assp: summability of b} holds with $p = 1$. If the Lie--Trotter splitting \eqref{eq: Lie splitting} is used in the QMC-TS method, then there exists a randomly shifted lattice rule \eqref{eq: QMC for F} that can be constructed by the CBC algorithm for any fixed $T$ such that the numerical solution $\psi_{m}^n$ given by the QMC-TS method for $N \le 10^{30}$ satisfies the following error estimate:
		\begin{align}\label{eq: main estimate}
			\sqrt{\mathbb{E}^{\bDelta}\left[\Big|\E[G(|\psi(t_n)|^2)] - Q_{m,N}(G(|\psi_{m}^n|^2); \bDelta)\Big|^2
				\right]} \leq C
			(m^{-\chi} + \tau + N^{- \kappa}),
		\end{align}
		for all $t_n = n \tau \in [0, T]$ and some constant $C > 0$ independent of $m, \tau, N$, where $\mathbb{E}^{\bDelta}$ denotes the expectation with respect to the random shift $\bDelta$, and $\kappa = 1/p - 1/2$ for $p \in (2/3, 1]$ and $\kappa = 1 - \delta$ for $p \in (0, 2/3]$ with $\delta > 0$ arbitrarily small.
	\end{theorem}

    \begin{remark}
        For a general numerical method, we need the condition $\psi_{\mathrm{in}}, v_0, v_j \in H^s(\bT)$ in Assumption \ref{assp: general} with $s$ greater than some certain value (which is usually greater than $1$) to obtain the temporal convergence, and the condition $\sum_{j = 1}^{\infty} a_j < \infty$ in Assumption \ref{assp: summability of b} to ensure that the temporal convergence is independent of the truncation dimension $m$. In particular, we need $s \ge 3$ for first-order temporal convergence of Lie--Trotter splitting. Moreover, high-order derivatives of the physical basis functions $v_j$ are usually more oscillatory than low-order ones in general. Therefore, the condition $\sum_{j = 1}^{\infty} a_j < \infty$ usually implies $\sum_{j = 1}^{\infty} b_{j}^p < \infty$ with a small $p$, and hence would lead to a QMC convergence rate close to $1$.
    \end{remark}

    \begin{remark} \label{rem: high-space-dimensional cases}
        Recall that we are focusing on the one-space-dimensional case. Stronger assumptions would be needed for the above theory to hold in high-space-dimensional cases; see Point (2) in Section \ref{subsec: discussion}.
    \end{remark}
	
	\section{Convergence analysis}
	\label{sec: convergence}
	
	\subsection{Preliminaries}
	
	We need some preliminary results for the proof of Theorem \ref{thm: main}. We will frequently use the algebraic property of $H^r(\bT)$ for $r > 1/2$ \cite[Chapter 4]{adams2021sobolev}, which reads that for any $f, g \in H^r(\bT)$
	\begin{align} \label{eq: algebraic property}
		\| fg \|_{H^r(\bT)} \le & C_{\mathrm{a}, r} \| f \|_{H^r(\bT)} \| g \|_{H^r(\bT)},
	\end{align}
	where the constant $C_{\mathrm{a}, r}$ depends on $r$, and the notation $C_{\mathrm{a}, r}$ is reserved for the constant in \eqref{eq: algebraic property} throughout this paper. Define the set $U_a := \{ \bxi \in \bR^\bN: \sum_{j = 1}^{\infty} a_{j} |\xi_j| < \infty \}$. Then, a minor modification of \cite[Lemma 2.28]{schwab2011sparse} gives the following lemma.
	\begin{lemma} \label{lem: measurability}
		Under Assumptions \ref{assp: general}--\ref{assp: summability of b}, we have $U_a \in \cB(\bR^\bN)$ and $\bmu_{\infty}(U_a) = 1$.
	\end{lemma}
	
	We also have the following lemma on the properties of the Gaussian random potentials.
	\begin{lemma} \label{lem: integrability of V}
		Let Assumptions \ref{assp: general}--\ref{assp: convergence of V} hold. Then, for any $q \ge 1$ and $m \in \bN^{+}$, we have $V_m \in L_{\bmu_{\infty}}^q(U, H^s(\bT))$, and $\| V_m \|_{L_{\bmu_{\infty}}^q(U, H^s(\bT))}$ can be bounded uniformly in $m$. Moreover, for $1 \le q \le 2 + \varepsilon$, we have $\lim \limits_{m \rightarrow \infty} \| V_m - V \|_{L_{\bmu_{\infty}}^q(U, H^1(\bT))} = 0$ and $V \in L_{\bmu_{\infty}}^q(U, H^1(\bT))$.
	\end{lemma}
	
	\begin{proof}
		Fix any $m \in \bN^{+}$. It is easy to see that $V_m(\bxi, \cdot) \in H^s(\bT)$ for any $\bxi \in U_a$, and hence $V_m(\bxi, \cdot) \in H^s(\bT)$ a.s. in $U$ by Lemma \ref{lem: measurability}. Then, for any $q \ge 1$, we have
		\begin{align*}
			\| V_m \|_{L_{\bmu_{\infty}}^q(U, H^s(\bT))}\le  \| v_0 \|_{H^s(\bT)} + \left( \int_{\bR} |\rho|^q \phi(\rho) \rd \rho \right)^{\frac{1}{q}} \sum_{j = 1}^{\infty} a_{j},
		\end{align*}
		where the upper bound on the right-hand side is finite due to Assumption \ref{assp: summability of b} and is independent of $m$. Hence, $V_m \in L_{\bmu_{\infty}}^q(U, H^s(\bT))$.
		
		On the other hand, for any $\bxi \in U_a$, we have $V_m(\bxi, \cdot) \rightarrow V(\bxi, \cdot)$ in $H^1(\bT)$ as $m \rightarrow \infty$, and hence $V(\bxi, \cdot) \in H^1(\bT)$. By Lemma \ref{lem: measurability}, we have $V(\bxi, \cdot) \in H^1(\bT)$ a.s. in $U$. Moreover, for any $1 \le q \le 2 + \varepsilon$, Assumption \ref{assp: convergence of V} gives $\lim \limits_{m \rightarrow \infty} \| V_m - V \|_{L_{\bmu_{\infty}}^q(U, H^1(\bT))} = 0$, and hence $V \in L_{\bmu_{\infty}}^q(U, H^1(\bT))$.
	\end{proof}
	
	In view of Lemmas \ref{lem: measurability}--\ref{lem: integrability of V}, we can define the solution to \eqref{eq: Schrodinger} for a.s. $\bxi \in U$ by the Duhamel's formula \cite{bers1964partial}, which reads
	\begin{align} \label{eq: solution by Duhamel's formula}
		\psi(t,\bxi, x) = \re^{\ri t \partial_x^2/2}\psi_{\mathrm{in}}(x) - \ri \int_0^t\re^{\ri (t-\rho) \partial_x^2/2}V(\bxi, x)\psi(\rho,\bxi, x) \rd \rho, \quad t\geq0.
	\end{align}
	We will also use the following results.
	
	\begin{lemma} \label{lem: integrability of exp V}
		Let Assumptions \ref{assp: general}--\ref{assp: convergence of V} hold. Then, we have $\exp(K \| V(\bxi) \|_{H^1(\bT)}) \in L_{\bmu_{\infty}}^q(U)$ for any $K > 0$ and $q \ge 1$.
	\end{lemma}
	
	\begin{proof}
		Lemma \ref{lem: integrability of V} indicates that $V_{\mathrm{r}}(\bxi, x)$ is an $H^1(\bT)$-valued centered Gaussian random variable, and $H^1(\bT)$ is a separable Hilbert space. Then, by Fernique's theorem \cite{da2014stochastic,fernique1975regularite} (see also \cite[Theorem 2.2]{charrier2012strong}), there exists a constant $\beta > 0$ such that $
		\int_U \exp \left( \beta \| V_{\mathrm{r}}(\bxi) \|_{H^1(\bT)}^2 \right) \rd \bmu_{\infty}(\bxi) < \infty$. Then, we have by Young's inequality that
		\begin{align*}
			& \int_{U} \exp(q K \| V(\bxi) \|_{H^1(\bT)}) \rd \bmu_{\infty}(\bxi) \\
            & \qquad \le \exp \left(q K \| v_0 \|_{H^1(\bT)} + q^2 K^2 / (4 \beta) \right) \int_U \exp \left( \beta  \| V_{\mathrm{r}}(\bxi) \|_{H^1(\bT)}^2 \right) \rd \bmu_{\infty}(\bxi) < \infty.
		\end{align*}
		Hence, for any $K > 0$ and $q \ge 1$, we have $\exp(K \| V(\bxi) \|_{H^1(\bT)}) \in L_{\bmu_{\infty}}^q(U)$.
	\end{proof}
	
	\begin{lemma} \label{lem: integrability of exp V m}
		Let Assumptions \ref{assp: general}--\ref{assp: summability of b} hold. Then, for any $K > 0$ and $q \ge 1$, it holds that $\exp(K \| V_m(\bxi) \|_{H^s(\bT)}) \in L_{\bmu_{\infty}}^q(U)$, and $\| \exp(K \| V_m(\bxi) \|_{H^s(\bT)}) \|_{L_{\bmu_{\infty}}^q(U)}$ can be bounded uniformly in $m$.
	\end{lemma}
	
	\begin{proof}
		We have from the proof of \cite[Theorem 16]{graham2015quasi} that for any $C \ge 0$
		\begin{align} \label{eq: property of density function}
			\int_{\bR} \exp(C |\xi|) \phi(\xi) \rd \xi = 2 \exp \left( \frac{C^2}{2} \right) \Phi(C), \quad \Phi(C) \le \frac{1}{2} \exp \left( \frac{2 C}{\sqrt{2 \pi}} \right).
		\end{align}
		Then,
		\begin{align*}
			\int_U \exp(q K \| V_m(\bxi) \|_{H^s(\bT)}) \rd \bmu_{\infty}(\bxi) \le & \exp(q K \| v_0 \|_{H^s(\bT)}) \prod_{j = 1}^m \int_{\bR} \exp(q K a_{j} |\xi_j|) \phi(\xi_j) \rd \xi_j \\
			\le & \exp \left( q K \| v_0 \|_{H^s(\bT)}  + \frac{2 q K}{\sqrt{2 \pi}} \sum_{j = 1}^{\infty} a_{j} + \frac{q^2 K^2}{2} \sum_{j = 1}^{\infty} a_{j}^2 \right),
		\end{align*}
		where the upper bound on the right-hand side is finite due to Assumption \ref{assp: summability of b} and is independent of $m$. Hence, we have $\exp(K \| V_m(\bxi) \|_{H^s(\bT)}) \in L_{\bmu_{\infty}}^q(U)$.
	\end{proof}
	
	\subsection{Dimension truncation error}
    \label{subsec: dimension truncation error}
	
	We first need the regularity of the solution in the physical domain.
	
	\begin{lemma} \label{lem: well-posedness of psi}
		Let Assumptions \ref{assp: general}--\ref{assp: convergence of V} hold.
        Then, for any $T > 0$ and $\bxi \in U_a$, we have
		\begin{align} \label{eq: well-posedness of psi}
			\|\psi(t,\bxi)\|_{H^1(\bT)} \le \|\psi_{\mathrm{in}}\|_{H^1(\bT)} \exp \left( C_{\mathrm{a}, 1} T \|V(\bxi)\|_{H^1(\bT)} \right), \quad 0 \le t \le T.
		\end{align}
        Moreover, for any $1 \le q < \infty$, we have $\psi \in L_{\bmu_{\infty}}^q(U, L^{\infty}((0, T), H^1(\bT)))$.
	\end{lemma}
	
	\begin{proof}
		Fix any $T > 0$ and $\bxi \in U_a$. By the algebraic property \eqref{eq: algebraic property} of $H^1(\bT)$ and the fact that $\re^{\ri t \partial_x^2 / 2}$ is an isometry on $H^1(\bT)$ for all $t \in \bR$, we take the $H^1$-norm on both sides of \eqref{eq: solution by Duhamel's formula} and obtain
		\begin{align} \label{eq: Hs norm of Duhamel's formula}
			\|\psi(t,\bxi)\|_{H^1(\bT)}\leq\|\psi_{\mathrm{in}}\|_{H^1(\bT)} + C_{\mathrm{a}, 1} \|V(\bxi)\|_{H^1(\bT)} \int_0^t
			\|\psi(\rho,\bxi)\|_{H^1(\bT)} \rd \rho.
		\end{align}
		Then, a bootstrap-type argument \cite{tao2006nonlinear} will give the local well-posedness of \eqref{eq: Schrodinger} in $H^1(\bT)$ for $\bxi$ (see also \cite[Appendix A]{wu2024error} for details of the bootstrap-type argument). Moreover, by Gronwall's inequality, we can deduce \eqref{eq: well-posedness of psi} from \eqref{eq: Hs norm of Duhamel's formula}.
		In addition, by Lemma \ref{lem: integrability of exp V}, we have $\psi \in L_{\bmu_{\infty}}^q(U, L^{\infty}((0, T), H^1(\bT)))$ for any $1 \le q < \infty$.
	\end{proof}
	
	\begin{lemma} \label{lem: well-posedness of psi m}
		Let Assumptions \ref{assp: general}--\ref{assp: summability of b} hold. Then, for any $T > 0$ and $\bxi \in U_a$, we have
		\begin{align} \label{eq: well-posedness of psi m}
			\|\psi_m(t,\bxi)\|_{H^s(\bT)} \le \|\psi_{\mathrm{in}}\|_{H^s(\bT)} \exp \left( C_{\mathrm{a}, s} T \|V_m(\bxi)\|_{H^s(\bT)} \right), \quad 0 \le t \le T.
		\end{align}
		Moreover, we have $\psi_m \in L_{\bmu_{\infty}}^q(U, L^{\infty}((0, T), H^s(\bT)))$, and $\| \psi_m \|_{L_{\bmu_{\infty}}^q(U, L^{\infty}((0, T), H^s(\bT)))}$ can be bounded uniformly in $m$ for any $1 \le q < \infty$.
	\end{lemma}
	
	\begin{proof}
		The proof uses \eqref{eq: algebraic property} and Lemma \ref{lem: integrability of exp V m}, and is similar to that of Lemma \ref{lem: well-posedness of psi}, so we omit it here.
	\end{proof}
	
	Now we give the dimension truncation error of the solution and the expectation of the physical observable.
	
	\begin{lemma} \label{lem: dimension truncation error}
		Under Assumptions \ref{assp: general}--\ref{assp: convergence of V}, we have for any $T > 0$
		\begin{align} \label{eq: dimension truncation error}
			\| \psi_m - \psi \|_{L_{\bmu_{\infty}}^2(U, L^{\infty}((0, T), H^1(\bT)))} \le C m^{- \chi},
		\end{align}
		where $C$ is independent of $m$.
	\end{lemma}
	
	\begin{proof}
		Let $\delta\psi = \psi_m - \psi$. Taking the difference between \eqref{eq: Schrodinger} and \eqref{eq: Schrodinger trun}, we have
		\begin{equation*}
			\left\{
			\begin{aligned}
				&\ri\partial_t\delta\psi=-\frac{1}{2}\partial_x^2\delta\psi+V
				\delta\psi+(V_m-V)\psi_m,\quad x\in \bT, \bxi \in U, t>0,\\
				&\delta\psi(t=0)=0,\quad x\in \bT.
			\end{aligned}
			\right.
		\end{equation*}
		For any $\bxi \in U_a$, the Duhamel's formula gives
		\begin{align*}
			\delta\psi(t,\bxi) = -\ri \int_0^t \re^{\ri(t-\rho) \partial_x^2/2} \left( V(\bxi) \delta\psi(\rho,\bxi) + (V_m(\bxi) - V(\bxi)) \psi_m(\rho,\bxi) \right) \rd \rho, \quad 0 \le t \le T,
		\end{align*}
		which by the algebraic property \eqref{eq: algebraic property} of $H^1(\bT)$ gives
		\begin{align*}
			\| \delta\psi(t,\bxi) \|_{H^1(\bT)} \le & C_{\mathrm{a}, 1} \| V(\bxi) \|_{H^1(\bT)} \int_0^t \| \delta\psi(\rho, \bxi) \|_{H^1(\bT)} \rd \rho \\
			& + C_{\mathrm{a}, 1} t \| \psi_m(\bxi) \|_{L^{\infty}((0, T), H^1(\bT))} \| V_m(\bxi) - V(\bxi) \|_{H^1(\bT)}, \quad 0 \le t \le T,
		\end{align*}
		where the constant $C_{\mathrm{a}, 1}$ is independent of $\bxi$. By Gronwall's inequality, we have for any $0 \le t \le T$ and $\bxi \in U_a$
		\begin{align} \label{eq: Hs norm of delta psi}
			\| \delta\psi(t,\bxi) \|_{H^1(\bT)} \le C_{\mathrm{a}, 1} T \| \psi_m(\bxi) \|_{L^{\infty}((0, T), H^1(\bT))} \| V_m(\bxi) - V(\bxi) \|_{H^1(\bT)} \exp \left( C_{\mathrm{a}, 1} T \| V(\bxi) \|_{H^1(\bT)} \right).
		\end{align}
		Then, \eqref{eq: Hs norm of delta psi} gives \eqref{eq: dimension truncation error} by Lemmas \ref{lem: integrability of exp V} and \ref{lem: well-posedness of psi m}, Assumption \ref{assp: convergence of V} and H\"{o}lder's inequality.
	\end{proof}
	
	\begin{lemma} \label{lem: dimension truncation of physical observable}
		Under Assumptions \ref{assp: general}--\ref{assp: convergence of V}, we have for any $T > 0$
		\begin{align} \label{eq: dimension truncation of physical observable}
			| \E[G(|\psi(t)|^2)] - \E[G(|\psi_m(t)|^2)] | \le C m^{-\chi}, \quad 0 \le t \le T,
		\end{align}
		where $C$ is independent of $m$.
	\end{lemma}
	
	\begin{proof}
		We have $\| |\psi(t)|^2 - |\psi_m(t)|^2 \|_{H^1(\bT)} \le C_{\mathrm{a}, 1} \| \psi(t) + \psi_m(t) \|_{H^1(\bT)} \| \psi(t) - \psi_m(t) \|_{H^1(\bT)}$ by the algebraic property \eqref{eq: algebraic property} of $H^1(\bT)$. Then,
		\begin{align*}
			| \E[G(|\psi(t)|^2)] - \E[G(|\psi_m(t)|^2)] | \le & \E[ | G(|\psi(t)|^2 - |\psi_m(t)|^2) | ] \\
			\le & \|G\|_{H^1(\bT)'} \E[\| |\psi(t)|^2 - |\psi_m(t)|^2 \|_{H^1(\bT)}] \\
			\le & C_{\mathrm{a}, 1} \|G\|_{H^1(\bT)'} \E[ \| \psi(t) + \psi_m(t) \|_{H^1(\bT)} \| \psi(t) - \psi_m(t) \|_{H^1(\bT)} ] \\
			\le & C_{\mathrm{a}, 1} \|G\|_{H^1(\bT)'} \left( \| \psi(t) \|_{L_{\bmu_{\infty}}^2(U, H^1(\bT))} + \| \psi_m(t) \|_{L_{\bmu_{\infty}}^2(U, H^1(\bT))} \right) \\
			& \times \| \psi(t) - \psi_m(t) \|_{L_{\bmu_{\infty}}^2(U, H^1(\bT))},
		\end{align*}
		where we have used the Cauchy--Schwarz inequality in the last inequality. The above equation gives \eqref{eq: dimension truncation of physical observable} by Lemmas \ref{lem: well-posedness of psi}--\ref{lem: dimension truncation error}.
	\end{proof}

    \begin{remark}
        As shown in the proofs of Lemmas \ref{lem: dimension truncation error}--\ref{lem: dimension truncation of physical observable}, by using Duhamel's formula, we find that the dimension truncation errors of the solution $\psi_m$ and the physical observable $G(|\psi_m|^2)$ would be controlled by that of the potential $V_m$. However, as will be shown in Example \ref{expl: extreme highD}, these error estimates for dimension truncation might not be optimal.
    \end{remark}
	
	\subsection{Temporal error}
	\label{subsec: temporal error}
	
	We first give two lemmas on the regularity of $\partial_t \psi_m$ and the semi-discrete solution $\psi_m^n$ in the physical domain, respectively.

	\begin{lemma} \label{lem: regularity of time derivative}
		Let Assumptions \ref{assp: general}--\ref{assp: summability of b} hold with $s \ge 3$. Then, for any $T > 0$ and $\bxi \in U_a$,
		\begin{align} \label{eq: regularity of time derivative}
			\| \partial_t \psi_m(t, \bxi) \|_{H^{s - 2}(\bT)} \le \| \psi_{\mathrm{in}} \|_{H^s(\bT)} + & (C_{\mathrm{a}, s - 2} \| V_m(\bxi) \|_{H^{s - 2}(\bT)} + T C_{\mathrm{a}, s} \| V_m(\bxi) \|_{H^{s}(\bT)}) \nonumber \\
            & \times \| \psi_{\mathrm{in}} \|_{H^s(\bT)} \exp \left( C_{\mathrm{a}, s} T \| V_m(\bxi) \|_{H^{s}(\bT)} \right)
		\end{align}
		where $0 \le t \le T$. Moreover, for any $1 \le q < \infty$, we have $\partial_t \psi_m \in L_{\bmu_{\infty}}^q(U, L^{\infty}((0, T), H^{s - 2}(\bT)))$ and $\| \partial_t \psi_m \|_{L_{\bmu_{\infty}}^q(U, L^{\infty}((0, T), H^{s - 2}(\bT)))}$ can be bounded uniformly in $m$.
	\end{lemma}

	\begin{proof}
        Fix any $T > 0$ and $\bxi \in U_a$. By Duhamel's formula, we have
        \begin{align*}
            \psi_m(t, \bxi, x) = \re^{\ri t \partial_x^2 / 2} \psi_{\mathrm{in}}(x) - \ri \int_0^t \re^{\ri (t - \rho) \partial_x^2 / 2} V_m(\bxi, x) \psi_m(\rho, \bxi, x) \rd \rho, \quad 0 \le t \le T.
        \end{align*}
        Taking the partial derivative of both sides of the above equation with respect to $t$, we have
        \begin{align*}
            \partial_t \psi_m(t, \bxi, x) = \frac{\ri}{2} \re^{\ri t \partial_x^2 / 2} \partial_x^2 \psi_{\mathrm{in}}(x) - \ri V_m(\bxi, x) \psi_m(t, \bxi, x) + \frac{1}{2} \int_0^t \re^{\ri (t - \rho) \partial_x^2 / 2} \partial_x^2 (V_m(\bxi, x) \psi_m(\rho, \bxi, x)) \rd \rho.
        \end{align*}
        Taking the $H^{s - 2}$-norm on both sides of the above equation, we have by the algebraic property \eqref{eq: algebraic property}
        \begin{align*}
            \| \partial_t \psi_m(t, \bxi) \|_{H^{s - 2}(\bT)} \le & \| \psi_{\mathrm{in}} \|_{H^s(\bT)} + C_{\mathrm{a}, s - 2} \| V_m(\bxi) \|_{H^{s - 2}(\bT)} \| \psi_m(t, \bxi) \|_{H^{s - 2}(\bT)} \\
            & + C_{\mathrm{a}, s} \| V_m(\bxi) \|_{H^s(\bT)} \int_0^t \| \psi(\rho, \bxi) \|_{H^s(\bT)} \rd \rho \\
            \le & \| \psi_{\mathrm{in}} \|_{H^s(\bT)} + C_{\mathrm{a}, s - 2} \| V_m(\bxi) \|_{H^{s - 2}(\bT)} \| \psi_m(t, \bxi) \|_{H^{s}(\bT)} \\
            & + C_{\mathrm{a}, s} \| V_m(\bxi) \|_{H^s(\bT)} \int_0^t \| \psi(\rho, \bxi) \|_{H^s(\bT)} \rd \rho.
        \end{align*}
        We obtain \eqref{eq: regularity of time derivative} by the above equation and Lemma \ref{lem: well-posedness of psi m}. Moreover, for any $1 \le q < \infty$, we have $\partial_t \psi_m \in L_{\bmu_{\infty}}^q(U, L^{\infty}((0, T), H^{s - 2}(\bT)))$ and $\| \partial_t \psi_m \|_{L_{\bmu_{\infty}}^q(U, L^{\infty}((0, T), H^{s - 2}(\bT)))}$ can be bounded uniformly in $m$, by Lemmas \ref{lem: integrability of V} and \ref{lem: integrability of exp V m} and H\"{o}lder's inequality.
	\end{proof}
	
	\begin{lemma} \label{lem: regularity of semidiscrete solution}
		Let Assumptions \ref{assp: general}--\ref{assp: summability of b} hold. Then, for any $0 < \tau \le T$ and $\bxi \in U_a$,
		\begin{align} \label{eq: regularity of semidiscrete solution}
			\| \psi_m^{n}(\bxi) \|_{H^s(\bT)} \le \| \psi_{\mathrm{in}} \|_{H^s(\bT)} \exp(C_{\mathrm{a}, s} T \| V_m(\bxi) \|_{H^s(\bT)}), \quad 0 \le n \le \lfloor T / \tau \rfloor.
		\end{align}
		Moreover, we have $\psi_m^n \in L_{\bmu_{\infty}}^q(U, H^s(\bT))$ and $\| \psi_m^n \|_{L_{\bmu_{\infty}}^q(U, H^s(\bT))}$ can be bounded uniformly in $m$ and $n$ for any $1 \le q < \infty$ and $0 \le n \le \lfloor T / \tau \rfloor$.
	\end{lemma}
	
	\begin{proof}
		Fix $T > 0$ and $\bxi \in U_a$. Recall $\psi_m^{n, *}$ defined in Remark \ref{rem: intermediate solution}. Since $\re^{\ri \tau \partial_x^2 / 2}$ is a linear isometry on $H^s(\bT)$, we have $\| \psi_{m}^{n} \|_{H^s(\bT)} = \| \psi_{m}^{n, *} \|_{H^s(\bT)}$. Furthermore, we consider $g_n$ introduced in Remark \ref{rem: intermediate solution}. Taking the $H^s(\bT)$-norm on both sides of \eqref{eq: intermediate solution}, we have by the algebraic property of $H^s(\bT)$ and Gronwall's inequality that $\| g(\tau) \|_{H^s(\bT)} \le \| g(0) \|_{H^s(\bT)} \exp(C_{\mathrm{a}, s} \tau \| V_m \|_{H^s(\bT)})$, i.e.,
		\begin{align} \label{eq: estimate of semi-discrete solution}
			\| \psi_m^{n}(\bxi) \|_{H^s(\bT)} = \| \psi_m^{n, *}(\bxi) \|_{H^s(\bT)} \le \| \psi_{m}^{n - 1}(\bxi) \|_{H^s(\bT)} \exp(C_{\mathrm{a}, s} \tau \| V_m(\bxi) \|_{H^s(\bT)}).
		\end{align}
		We have \eqref{eq: regularity of semidiscrete solution} by iterating \eqref{eq: estimate of semi-discrete solution} from $n$ to $1$. Moreover, for any $1 \le q < \infty$ and $0 \le n \le \lfloor T / \tau \rfloor$, we have by Lemma \ref{lem: integrability of exp V m} $\psi_m^n \in L_{\bmu_{\infty}}^q(U, H^s(\bT))$ and $\| \psi_m^n \|_{L_{\bmu_{\infty}}^q(U, H^s(\bT))}$ can be bounded independently of $m$ and $n$.
	\end{proof}
	
	Now we can give the temporal error of the solution and the expectation of the physical observable.
	
	\begin{lemma} \label{lem: temporal error of solution}
		Let Assumptions \ref{assp: general}--\ref{assp: summability of b} hold with $s \ge 3$. Then, for any $0 < \tau \le T$,
		\begin{align} \label{eq: temporal error of solution}
			\| \psi_m(t_n) - \psi_m^n \|_{L_{\bmu_{\infty}}^2(U, H^{s - 2}(\bT))} \le C \tau, \quad 0 \le t_n = n \tau \le T,
		\end{align}
		where $C$ is independent of $m$ and $\tau$.
	\end{lemma}
	
	\begin{proof}
		The proof largely follows the analysis in \cite[Section 3]{su2020time} for the Lie--Trotter splitting. For any $f, g \in H^{s - 2}(\bT)$, we have
		\begin{align*}
			\| \Psi^\mathrm{k}_{\tau} \circ \Psi^\mathrm{p}_{\tau}(f) - \Psi^\mathrm{k}_{\tau} \circ \Psi^\mathrm{p}_{\tau}(g) \|_{H^{s - 2}(\bT)} = \| \Psi^\mathrm{p}_{\tau}(f) - \Psi^\mathrm{p}_{\tau}(g) \|_{H^{s - 2}(\bT)},
		\end{align*}
		since $\re^{\ri \tau \partial_x^2 / 2}$ is an isometry on $H^{s - 2}(\bT)$. Let $\widetilde{f}(t) = \Psi^\mathrm{p}_{t}(f)$ and $\widetilde{g}(t) = \Psi^\mathrm{p}_{t}(g)$. Then, by the Duhamel's formula and the algebraic property of $H^{s - 2}(\bT)$, we have for $0 \le t \le \tau$
		\begin{align*}
			\| \widetilde{f}(t) - \widetilde{g}(t) \|_{H^{s - 2}(\bT)} \le & \| f - g \|_{H^{s - 2}(\bT)} + C_{\mathrm{a}, s - 2} \int_0^{t} \| V \|_{H^{s - 2}(\bT)} \| \widetilde{f}(\rho) - \widetilde{g}(\rho) \|_{H^{s - 2}(\bT)} \rd \rho.
		\end{align*}
		By Gronwall's inequality, we have
		\begin{align*}
			\| \Psi^\mathrm{k}_{\tau} \circ \Psi^\mathrm{p}_{\tau}(f) - \Psi^\mathrm{k}_{\tau} \circ \Psi^\mathrm{p}_{\tau}(g) \|_{H^{s - 2}(\bT)} = \| \widetilde{f}(\tau) - \widetilde{g}(\tau) \|_{H^{s - 2}(\bT)} \le \re^{C_{\mathrm{a}, s - 2} \tau \| V \|_{H^{s - 2}(\bT)}} \| f - g \|_{H^{s - 2}(\bT)}.
		\end{align*}
		
		On the other hand, by Taylor expansion, we have
		\begin{align*}
			\Psi^\mathrm{k}_{\tau} \circ \Psi^\mathrm{p}_{\tau}(\psi_m(t_{n - 1})) = & \re^{\ri \tau \partial_x^2 / 2} \re^{-\ri \tau V_m} \psi_m(t_{n - 1}) \\
			= & \re^{\ri \tau \partial_x^2 / 2} \left( 1 - \ri \tau V_m - \tau^2 V_m^2 \int_0^1 (1 - \theta) \re^{- \ri \theta \tau V_m} \rd \theta \right) \psi_m(t_{n - 1}).
		\end{align*}
		Moreover, Duhamel's formula gives
		\begin{align*}
			\psi_m(t_n) = & \re^{\ri \tau \partial_x^2 / 2} \psi_m(t_{n - 1}) - \ri \int_0^{\tau} \re^{\ri (\tau - \rho) \partial_x^2 / 2} V_m \psi_m(t_{n - 1} + \rho) \rd \rho \\
			= & \re^{\ri \tau \partial_x^2 / 2} \psi_m(t_{n - 1}) - \ri \tau \re^{\ri \tau \partial_x^2 / 2} V_m \psi_m(t_{n - 1}) \\
			& - \re^{\ri \tau \partial_x^2 / 2} \int_0^{\tau} \int_0^1 \re^{- \ri \theta \rho \partial_x^2 / 2} \rd \theta ( \rho \partial_x^2 / 2 ) [ V_m \psi_m(t_{n - 1} + \rho) ] \rd \rho \\
			& - \ri \re^{\ri \tau \partial_x^2 / 2} \int_0^{\tau} V_m (\psi_m(t_{n - 1} + \rho) - \psi_m(t_{n - 1})) \rd \rho.
		\end{align*}
		Then, the local error reads
		\begin{align*}
			& \| \psi_m(t_n, \bxi) - \Psi^\mathrm{k}_{\tau} \circ \Psi^\mathrm{p}_{\tau}(\psi_m(t_{n - 1}, \bxi), \bxi) \|_{H^{s - 2}(\bT)} \\
			& \qquad \le (C_{\mathrm{a}, s - 2})^3 \tau^2 \| V_m(\bxi) \|_{H^{s - 2}(\bT)}^2 \| \psi_m(t_{n - 1}, \bxi) \|_{H^{s - 2}(\bT)} \sup_{0 \le \theta \le 1} \| \re^{- \ri \theta \tau V_m(\bxi)} \|_{H^{s - 2}(\bT)} \\
			& \qquad \quad + C_{\mathrm{a}, s} \tau^2 \sup_{0 \le \rho \le \tau} \| \psi_m(t_{n - 1} + \rho, \bxi) \|_{H^s(\bT)} \| V_m(\bxi) \|_{H^s(\bT)} \\
			& \qquad \quad + C_{\mathrm{a}, s - 2} \| V_m(\bxi) \|_{H^{s - 2}(\bT)} \int_0^{\tau} \| \psi_m(t_{n - 1} + \rho, \bxi) - \psi_m(t_{n - 1}, \bxi) \|_{H^{s - 2}(\bT)} \rd \rho,
		\end{align*}
		where the constants $C_{\mathrm{a}, s - 2}$ and $C_{\mathrm{a}, s}$ come from the algebraic property \eqref{eq: algebraic property} of $H^{s - 2}(\bT)$ and $H^s(\bT)$, respectively. Fa\`{a} di Bruno's formula (see, e.g., \cite{hardy2006combinatorics}) gives
		\begin{align*}
			&\sup_{0 \le \theta \le 1} \| \re^{- \ri \theta \tau V_m(\bxi)} \|_{H^{s - 2}(\bT)} \\
			& \quad \le \sup_{0 \le \theta \le 1} \left( |\bT|^{1/2} + \sum_{j = 1}^{s - 2} \| \partial_x^j \re^{- \ri \theta \tau V_m(\bxi)} \|_{L^{2}(\bT)} \right) \\
            & \quad \le \sup_{0 \le \theta \le 1} \left( |\bT|^{1/2} + \sum_{j = 1}^{s - 2}\sum_{k = 1}^{j} \| \exp(- \ri \theta \tau V_m(\bxi)) B_{j, k}(- \ri \theta \tau \partial_x V_m(\bxi), - \ri \theta \tau \partial_x^2 V_m(\bxi), \ldots, - \ri \theta \tau \partial_x^{j - k + 1} V_m(\bxi)) \|_{L^2(\bT)} \right) \\
			& \quad = \sup_{0 \le \theta \le 1} \left( |\bT|^{1/2} + \sum_{j = 1}^{s - 2}\sum_{k = 1}^{j} \| B_{j, k}(- \ri \theta \tau \partial_x V_m(\bxi), - \ri \theta \tau \partial_x^2 V_m(\bxi), \ldots, - \ri \theta \tau \partial_x^{j - k + 1} V_m(\bxi)) \|_{L^2(\bT)} \right) \\
			& \quad =: K_1(\bxi),
		\end{align*}
		where $|\bT|$ is the Lebesgue measure of $\bT$, the last equality is due the fact that $\exp(- \ri \theta \tau V_m(\bxi))$ is an isometry on $L^2(\bT)$, and
        $\{ B_{j, k} \}_{k = 1}^{j}$, $j = 1, \ldots, s - 2$, are the Bell polynomials \cite{bell1934exponential}. Here, $K_1(\bxi)$ can be bounded by a polynomial involving $(\| V_m(\bxi) \|_{H^j(\bT)})^{\ell}$, $j = 1, \ldots, s - 2$, $\ell = 1, \ldots, s - j - 1$, by the algebraic property \eqref{eq: algebraic property}. Since $V_m \in L_{\bmu_{\infty}}^q(U, H^s(\bT))$ and $\| V_m \|_{L_{\bmu_{\infty}}^q(U, H^s(\bT))}$ can be bounded uniformly in $m$ for any $1 \le q < \infty$ by Lemma \ref{lem: integrability of V}, we have by H\"{o}lder's inequality that $(\| V_m(\bxi) \|_{H^j(\bT)})^{\ell} \in L_{\bmu_{\infty}}^q(U)$ and $\| (\| V_m(\bxi) \|_{H^j(\bT)})^{\ell} \|_{L_{\bmu_{\infty}}^q(U)}$ can be uniformly bounded in $m$ for any $1 \le q < \infty$. Therefore, $K_1(\bxi) \in L_{\bmu_{\infty}}^q(U)$ and $\| K_1 \|_{L_{\bmu_{\infty}}^q(U)}$ can be bounded uniformly in $m$ and $\tau \in (0, T]$ for any $1 \le q < \infty$. Some additional calculations give
		\begin{align*}
			\int_0^{\tau} \| \psi_m(t_{n - 1} + \rho) - \psi_m(t_{n - 1}) \|_{H^{s - 2}(\bT)} = & \int_0^{\tau} \left\| \int_0^{\rho} \partial_t \psi_m(t_{n - 1} + y) \rd y \right\|_{H^{s - 2}(\bT)} \rd \rho \\
			\le & \tau^2 \sup_{0 \le \rho \le \tau} \| \partial_t \psi_m(t_{n - 1} + \rho) \|_{H^{s - 2}(\bT)}.
		\end{align*}
		Therefore, we have
		\begin{align*}
			\| \psi_m(t_n, \bxi) - \Psi^\mathrm{k}_{\tau} \circ \Psi^\mathrm{p}_{\tau}(\psi_m(t_{n - 1}, \bxi), \bxi) \|_{H^{s - 2}(\bT)} \le K_2(\bxi) \tau^2,
		\end{align*}
		where
		\begin{align*}
			K_2(\bxi) = & \tau^2 \Big( (C_{\mathrm{a}, s - 2})^3 K_1(\bxi) \| V_m(\bxi) \|_{H^{s - 2}(\bT)}^2 \| \psi_m(\bxi) \|_{L^{\infty}((0, T), H^{s - 2}(\bT))} \\
			& \qquad + C_{\mathrm{a}, s} \| V_m(\bxi) \|_{H^s(\bT)} \| \psi_m(\bxi) \|_{L^{\infty}((0, T), H^s(\bT))} \\
			& \qquad + C_{\mathrm{a}, s - 2} \| V_m(\bxi) \|_{H^{s - 2}(\bT)} \| \partial_t \psi_m(\bxi) \|_{L^{\infty}((0, T), H^{s - 2}(\bT))} \Big).
		\end{align*}
		By Lemmas \ref{lem: integrability of V}, \ref{lem: well-posedness of psi m} and \ref{lem: regularity of time derivative} and H\"{o}lder's inequality, $K_2(\bxi) \in L_{\bmu_{\infty}}^q(U)$ and $\| K_2 \|_{L_{\bmu_{\infty}}^q(U)}$ can be bounded uniformly in $m$ and $\tau$ for any $1 \le q < \infty$.
		
		Finally, we have
		\begin{align*}
			& \| \psi_m(t_n, \bxi) - \psi_m^n(\bxi) \|_{H^{s - 2}(\bT)} \\
			& \qquad \le \| \psi_m(t_n, \bxi) - \Psi^\mathrm{k}_{\tau} \circ \Psi^\mathrm{p}_{\tau}(\psi_m(t_{n - 1}, \bxi), \bxi) \|_{H^{s - 2}(\bT)} \\
			& \qquad \quad+ \| \Psi^\mathrm{k}_{\tau} \circ \Psi^\mathrm{p}_{\tau}(\psi_m(t_{n - 1}, \bxi), \bxi) - \Psi^\mathrm{k}_{\tau} \circ \Psi^\mathrm{p}_{\tau}(\psi_m^{n - 1}(\bxi), \bxi) \|_{H^{s - 2}(\bT)} \\
			& \qquad \le \re^{C_{\mathrm{a}, s - 2} \tau \| V_m(\bxi) \|_{H^{s - 2}(\bT)}} \| \psi_m(t_{n - 1}, \bxi) - \psi_m^{n - 1}(\bxi) \|_{H^{s - 2}(\bT)} + K_2(\bxi) \tau^2 \\
			& \qquad \le \re^{C_{\mathrm{a}, s - 2} T \| V_m(\bxi) \|_{H^{s - 2}(\bT)}} \| \psi_m(0, \bxi) - \psi_m^{0}(\bxi) \|_{H^{s - 2}(\bT)} + K_2(\bxi) \tau^2 \sum_{j = 0}^{n - 1} \re^{C_{\mathrm{a}, s - 2} j \tau \| V_m(\bxi) \|_{H^{s - 2}(\bT)}} \\
			& \qquad \le T \re^{C_{\mathrm{a}, s - 2} T \| V_m(\bxi) \|_{H^{s - 2}(\bT)}} K_2(\bxi) \tau,
		\end{align*}
		and we can obtain \eqref{eq: temporal error of solution} by Lemma \ref{lem: integrability of exp V m} and H\"{o}lder's inequality.
	\end{proof}
	
	\begin{lemma} \label{lem: temporal error of functional}
		Let Assumptions \ref{assp: general}--\ref{assp: summability of b} hold with $s \ge 3$. Then, for any $0 < \tau \le T$,
		\begin{align} \label{eq: temporal error of functional}
			| \E[ G( | \psi_m(t_n) |^2 ) ] - \E[ G( | \psi_m^n |^2 ) ] | \le C \tau, \quad 0 \le t_n = n \tau \le T,
		\end{align}
		where $C$ is independent of $m$ and $\tau$.
	\end{lemma}
	
	\begin{proof}
		The proof uses Lemmas \ref{lem: well-posedness of psi m}, \ref{lem: regularity of semidiscrete solution} and \ref{lem: temporal error of solution}, and is similar to that of Lemma \ref{lem: dimension truncation of physical observable}. Hence, we omit it here.
	\end{proof}
	
	\subsection{QMC quadrature error}
	\label{subsec: QMC error}
	
	The analysis of the QMC quadrature error is closely related to the mixed first derivatives of $\psi_m^n$ with respect to $\xi_m$, whose growth as $|\bxi_m| \rightarrow \infty$ has a direct impact on the convergence rate of QMC. Let $\nu = (\nu_1, \nu_2, \ldots, \nu_m) \in \cJ := \{0, 1\}^m$ with $| \nu | = \sum_{j = 1}^m \nu_j$, $\cI_{\nu} = \{ j: \nu_j = 1 \}$, and $\cI_{\nu}^{\mathsf{c}} = \{ 1, 2, \ldots, m \} \setminus \cI_{\nu}$. The following notation is adopted to denote the mixed first derivative:
	\begin{align*}
		\partial^{\nu} \psi_m^n = \partial_{\xi_1}^{\nu_1}\ldots\partial_{\xi_m}^{\nu_m} \psi_m^n, \quad  \nu \in \cJ.
	\end{align*}
	
	Before proceeding further, we give bounds on the $L^2(\bT)$-norms of $g_n$ and $\partial_x g_n$ introduced in Remark \ref{rem: intermediate solution}.
	
	\begin{lemma} \label{lem: physical regularity of intermediate solution}
		Under Assumptions \ref{assp: general}--\ref{assp: summability of b}, we have for any $0 < \tau \le T$ and $\bxi \in U_a$
		\begin{align}
			\| g_n(\rho, \bxi) \|_{L^2(\bT)} = & \| \psi_{\mathrm{in}} \|_{L^2(\bT)}, \label{eq: physical regularity of intermediate solution} \\
			\| \partial_x g_n(\rho, \bxi) \|_{L^2(\bT)} \le & \| \partial_x \psi_{\mathrm{in}} \|_{L^2(\bT)} + T \| \partial_x V_m(\bxi) \|_{L^{\infty}(\bT)} \| \psi_{\mathrm{in}} \|_{L^2(\bT)}, \label{eq: physical regularity of partial x intermediate solution}
		\end{align}
		where $0 \le \rho \le \tau$ and $1 \le n \le \lfloor T / \tau \rfloor$.
	\end{lemma}
	
	\begin{proof}
		Fix $0 \le \rho \le \tau \le T, 1 \le n \le \lfloor T / \tau \rfloor$ and $\bxi \in U_a$. First, we have $g_n(\rho) = \re^{- \ri \rho V_m} \psi_m^{n - 1} = \re^{- \ri \rho V_m} \re^{\ri \tau \partial_x^2 / 2} g_{n - 1}(\tau)$, and thus
		\begin{align*}
			\| g_n(\rho) \|_{L^2(\bT)} = \| g_{n - 1}(\tau) \|_{L^2(\bT)} = \cdots = \| g_{1}(\tau) \|_{L^2(\bT)} = \| \psi_{\mathrm{in}} \|_{L^2(\bT)}.
		\end{align*}
		
		On the other hand, taking the partial derivative of \eqref{eq: intermediate solution} with respect to $x$, we have
		\begin{align*}
			\ri \partial_t \partial_x g_n = V_m (\partial_x g_n) + (\partial_x V_m) g_n.
		\end{align*}
		Multiplying the above equation by $\overline{\partial_x g_n}$, integrating it with respect to $x$ over $\bT$, and taking the imaginary part of it, we obtain
		\begin{align*}
			\frac{1}{2} \frac{\rd}{\rd t} \| \partial_x g_n \|_{L^2(\bT)}^2 = \mathrm{Im} \left( \int_{\bT} (\overline{ \partial_x g_n }) (\partial_x V_m) g_n \right).
		\end{align*}
		The Cauchy--Schwarz inequality gives
		\begin{align*}
			\| \partial_x g_n \|_{L^2(\bT)} \frac{\rd}{\rd t} \| \partial_x g_n \|_{L^2(\bT)} \le \| \partial_x g_n \|_{L^2(\bT)} \| \partial_x V_m \|_{L^\infty(\bT)} \| g_n \|_{L^2(\bT)}.
		\end{align*}
		Dropping the term of $\| \partial_x g_n \|_{L^2(\bT)}$ in the above equation, integrating it with respect to $t$ over $(0, \rho)$ and using \eqref{eq: physical regularity of intermediate solution}, we have
		\begin{align} \label{eq: partial x g zero}
			\| \partial_x g_n(\rho) \|_{L^2(\bT)} \le \| \partial_x g_n(0) \|_{L^2(\bT)} + \rho \| \partial_x V_m \|_{L^{\infty}(\bT)} \| \psi_{\mathrm{in}} \|_{L^2(\bT)}.
		\end{align}
		Note that $\| \partial_x \re^{\ri \tau \partial_x^2 / 2} f \|_{L^2(\bT)} = \| \partial_x f \|_{L^2(\bT)}$ for any $f \in H^1(\bT)$. Then, we deduce from \eqref{eq: partial x g zero}
		\begin{align} \label{eq: partial x g tau}
			\| \partial_x g_n(\rho) \|_{L^2(\bT)} \le \| \partial_x g_{n - 1}(\tau) \|_{L^2(\bT)} + \tau \| \partial_x V_m \|_{L^{\infty}(\bT)} \| \psi_{\mathrm{in}} \|_{L^2(\bT)}.
		\end{align}
		Finally, we obtain \eqref{eq: physical regularity of partial x intermediate solution} by iterating \eqref{eq: partial x g tau} from $n$ to $2$ and using \eqref{eq: partial x g zero} for $g_1$.
	\end{proof}
	
	Let $C_{T} = \max\{ 1, T \}$ and $\Upsilon_m(\bxi) = \max\{ 1, \| V_m(\bxi) \|_{W^{1, \infty}(\bT)} \}$. Now we give bounds on $\partial^{\nu} \psi_m^n$ and $\partial^{\nu} G(|\psi_m^n|^2)$.
	
	\begin{lemma} \label{lem: parametric regularity of semidiscrete in H1}
		Under Assumptions \ref{assp: general}--\ref{assp: summability of b}, we have for any $0 < \tau \le T$ and $\bxi \in U_a$
		\begin{align} \label{eq: parametric regularity of semidiscrete in H1}
			\| \partial^{\nu} \psi_m^n(\bxi) \|_{H^1(\bT)} \le |\nu|! 4^{|\nu| + 1} C_{T}^{|\nu| + 1} \| \psi_{\mathrm{in}} \|_{H^1(\bT)} \Upsilon_m(\bxi) \prod_{j \in \cI_{\nu}} b_j, \quad 1 \le n \le \lfloor T / \tau \rfloor.
		\end{align}
	\end{lemma}
	
	\begin{proof}
		Fix $T > 0$ and $\bxi \in U_a$. We shall prove that for any $0 \le n \le \lfloor T / \tau \rfloor$ and $0 \le \rho \le \tau$
		\begin{align}
			\| \partial^{\nu} g_n(\rho, \bxi) \|_{L^2(\bT)} & \le |\nu|! C_{T}^{|\nu|} \| \psi_{\mathrm{in}} \|_{L^2(\bT)} \prod_{j \in \cI_{\nu}} b_j, \label{eq: parametric regularity of psi} \\
			\| \partial_x \partial^{\nu} g_n(\rho, \bxi) \|_{L^2(\bT)} & \le |\nu|! 4^{|\nu|} C_{T}^{|\nu| + 1} \| \psi_{\mathrm{in}} \|_{H^1(\bT)} \Upsilon_m(\bxi) \prod_{j \in \cI_{\nu}} b_j, \label{eq: parametric regularity of partial x psi}
		\end{align}
		and then \eqref{eq: parametric regularity of semidiscrete in H1} follows from the above two equations and the relations that $\| \partial^{\nu} \psi_{m}^{n} \|_{L^2(\bT)} = \| \partial^{\nu} g_n(\tau) \|_{L^2(\bT)}$ and $\| \partial_x \partial^{\nu} \psi_{m}^{n} \|_{L^2(\bT)} = \| \partial_x \partial^{\nu} g_n(\tau) \|_{L^2(\bT)}$ for any $|\nu| \ge 1$, which are due to the fact that $\partial^{\nu} \psi_{m}^{n} = \re^{\ri \tau \partial_x^2 / 2} \partial^{\nu} g_n(\tau)$. Now we prove \eqref{eq: parametric regularity of psi} and \eqref{eq: parametric regularity of partial x psi} by induction on $|\nu|$.
		
		For $|\nu| = 1$, we differentiate \eqref{eq: intermediate solution} with respect to $\xi_j$ and obtain
		\begin{align*}
			\ri \partial_t \partial_j g_n = V_m (\partial_j g_n) + (\partial_j V_m) g_n,
		\end{align*}
		where we write $\partial_j = \partial_{\xi_j}$ for short. Similar arguments to the proof of \eqref{eq: physical regularity of partial x intermediate solution} in Lemma \ref{lem: physical regularity of intermediate solution} give
		\begin{align*}
			\| \partial_j g_n(\rho) \|_{L^2(\bT)} \le & \| \partial_j g_n(0) \|_{L^2(\bT)} + \int_0^{\rho} \| \partial_j V_m \|_{L^{\infty}(\bT)} \| g_n(y) \|_{L^2(\bT)} \rd y \\
			\le & \| \partial_j g_{n - 1}(\tau) \|_{L^2(\bT)} + \tau \lambda_j \| v_j \|_{L^{\infty}(\bT)} \| \psi_{\mathrm{in}} \|_{L^2(\bT)},
		\end{align*}
		where we have used \eqref{eq: physical regularity of intermediate solution} and the fact that $g_n(0) = \psi_m^{n - 1} = \re^{\ri \tau \partial_x^2 / 2} g_{n - 1}(\tau)$ in the second inequality. We obtain \eqref{eq: parametric regularity of psi} for $|\nu| = 1$ by iterating the above equation from $n$ to $2$ and the fact that $\partial_j g_1(0) = \partial_j \psi_{\mathrm{in}} = 0$. Then, differentiating \eqref{eq: intermediate solution} with respect to $\xi_j$ and $x$, we have
		\begin{align*}
			\ri \partial_t \partial_x \partial_j g_n = V_m (\partial_x \partial_j g_n) + (\partial_x V_m) (\partial_j g_n) + (\partial_j V_m) (\partial_x g_n) + (\partial_x \partial_j V_m) g_n.
		\end{align*}
		Following a similar procedure, we obtain
		\begin{align*}
			\| \partial_x \partial_j g_n(\tau) \|_{L^2(\bT)} \le & \| \partial_x \partial_j g_n(0) \|_{L^2(\bT)} + \int_0^{\tau} \Big{(} \| \partial_x V_m \|_{L^{\infty}(\bT)} \| \partial_j g_n(y) \|_{L^2(\bT)} \\
			& + \| \partial_j V_m \|_{L^{\infty}(\bT)} \| \partial_x g_n(y) \|_{L^2(\bT)} + \| \partial_x \partial_j V_m \|_{L^{\infty}(\bT)} \| g_n(y) \|_{L^2(\bT)}\Big{)} \rd y \\
			\le & \| \partial_x \partial_j g_{n - 1}(\tau) \|_{L^2(\bT)} + C_T \tau b_j \| \partial_x V_m \|_{L^{\infty}(\bT)} \| \psi_{\mathrm{in}} \|_{L^2(\bT)} \\
			& + \tau \lambda_j \| v_j \|_{L^{\infty}(\bT)} \| \partial_x \psi_{\mathrm{in}} \|_{L^2(\bT)} + T \tau \lambda_j \| v_j \|_{L^{\infty}(\bT)} \| \partial_x V_m \|_{L^{\infty}(\bT)} \| \psi_{\mathrm{in}} \|_{L^2(\bT)} \\
			& + \tau \lambda_j \| \partial_x v_j \|_{L^{\infty}(\bT)} \| \psi_{\mathrm{in}} \|_{L^2(\bT)} \\
			\le & \| \partial_x \partial_j g_{n - 1}(\tau) \|_{L^2(\bT)} + 4 \tau C_T \|
			\psi_{\mathrm{in}} \|_{H^1(\bT)} \Upsilon_m b_j,
		\end{align*}
		where we have used Lemma \ref{lem: physical regularity of intermediate solution}, \eqref{eq: parametric regularity of psi} for $|\nu| = 1$ and the fact that $g_n(0) = \psi_m^{n - 1} = \re^{\ri \tau \partial_x^2 / 2} g_{n - 1}(\tau)$ in the second inequality. We obtain \eqref{eq: parametric regularity of partial x psi} for $|\nu| = 1$ by iterating the above equation from $n$ to $2$ and the fact that $\partial_x \partial_j g_1(0) = \partial_x \partial_j \psi_{\mathrm{in}} = 0$.

		For $|\nu| \ge 2$, the Leibniz rule gives
		\begin{align*}
			\ri \partial_t \partial^{\nu} g_n = & \sum_{\mu \le \nu} \binom{{\nu}}{\mu} (\partial^{{\nu}-{\mu}}V_m) (\partial^{\mu} g_n) \\
            = & V_m (\partial^{\nu} g_n) + \sum_{j \in \{j: \nu_j = 1\}} (\partial_j V_m) (\partial^{\nu - e_j} g_n) = V_m (\partial^{\nu} g_n) + \sum_{j \in \{j: \nu_j = 1\}} \lambda_j v_j (\partial^{\nu - e_j} g_n),
        \end{align*}
        where $\le$ denotes the component-wise order relation for multi-indices, i.e., $\mu \le \nu \iff  \mu_j \le \nu_j$ for all $j$, $\binom{\nu}{\mu}=\prod_{j=1}^{m}\binom{\nu_j}{\mu_j}$, $e_j$ is the $j$th unit basis vector, and the second equality is due to the fact that $\partial^{\nu} V_m = 0$ for $|\nu| \ge 2$. We can deduce in the aforementioned way that
        \begin{align}
            \| \partial^{\nu} g_n (\rho) \|_{L^2(\bT)} \le \| \partial^{\nu} g_{n - 1} (\tau) \|_{L^2(\bT)} + \sum_{j \in \{j: \nu_j = 1\}} \int_0^{\tau} b_j \| \partial^{\nu - e_j} g_n(y) \|_{L^2(\bT)} \rd y. \label{eq: intermediate high estimate of psi m}
        \end{align}
        Note that $|\nu - e_j| = |\nu| - 1$ and $\#\{ j: \nu_j = 1 \} = |\nu|$. Then, inserting the assumption \eqref{eq: parametric regularity of psi} for $\partial^{\nu - e_j} g_n$ into \eqref{eq: intermediate high estimate of psi m}, we can obtain \eqref{eq: parametric regularity of psi} for $|\nu|$ by iterating \eqref{eq: intermediate high estimate of psi m} from $n$ to $2$. Similarly, for $ \partial_x \partial^{\nu} g_n$, we have
        \begin{align*}
			\ri \partial_t \partial_x \partial^{\nu} g_n = & \sum_{\mu \le \nu} \binom{{\nu}}{\mu} ( (\partial^{{\nu}-{\mu}} V_m) (\partial_x \partial^{\mu} g_n) + (\partial_x \partial^{{\nu}-{\mu}} V_m) (\partial^{\mu} g_n) ) \\
            = & V_m (\partial_x \partial^{\nu} g_n) + (\partial_x V_m) (\partial^{\nu} g_n) + \sum_{j \in \{j: \nu_j = 1\}} ( (\partial_j V_m) (\partial_x \partial^{\nu - e_j} g_n) + (\partial_x \partial_j V_m) (\partial^{\nu - e_j} g_n) ) \\
            = & V_m (\partial_x \partial^{\nu} g_n) + (\partial_x V_m) (\partial^{\nu} g_n) + \sum_{j \in \{j: \nu_j = 1\}} ( \lambda_j v_j (\partial_x \partial^{\nu - e_j} g_n) + (\lambda_j \partial_x v_j) (\partial^{\nu - e_j} g_n) ).
		\end{align*}
		We can then deduce that
		\begin{align}
			\| \partial_x  \partial^{\nu} g_n(\rho) \|_{L^2(\bT)} \le & \| \partial_x  \partial^{\nu} g_{n - 1}(\tau) \|_{L^2(\bT)} + \int_0^{\tau} \Bigg{(} \| \partial_x V_m \|_{L^{\infty}(\bT)} \| \partial^{\nu} g_n(y) \|_{L^2(\bT)} \nonumber \\
			& + \sum_{j \in \{j: \nu_j = 1\}} \big{(} b_j \| \partial_x \partial^{\nu - e_j} g_n(y) \|_{L^2(\bT)} + b_j \| \partial^{\nu - e_j} g_n(y) \|_{L^2(\bT)} \big{)} \Bigg{)} \rd y. \label{eq: intermediate high estimate of partial x psi m}
		\end{align}
		Finally, inserting \eqref{eq: parametric regularity of psi} for $\partial^{\nu} g_n, \partial^{\nu - e_j} g_n$ and the assumption \eqref{eq: parametric regularity of partial x psi} for $\partial_x \partial^{\nu - e_j} g_n$ into \eqref{eq: intermediate high estimate of partial x psi m}, we obtain
        \begin{align*}
            \| \partial_x \partial^{\nu} g_n(\rho) \|_{L^2(\bT)} \le & \| \partial_x  \partial^{\nu} g_{n - 1}(\tau) \|_{L^2(\bT)} + \tau |\nu|! \| \psi_{\mathrm{in}} \|_{L^2(\bT)} \prod_{j \in \cI_{\nu}} b_j \\
            & \times (C_{T}^{|\nu|} \| \partial_x V_m \|_{L^{\infty}(\bT)} + 4^{|\nu| - 1} C_{T}^{|\nu|} \Upsilon_m(\bxi) + C_{T}^{|\nu| - 1}) \\
            \le & \| \partial_x  \partial^{\nu} g_{n - 1}(\tau) \|_{L^2(\bT)} + \tau |\nu|! 4^{|\nu|} C_{T}^{|\nu|} \| \psi_{\mathrm{in}} \|_{H^1(\bT)} \Upsilon_m(\bxi) \prod_{j \in \cI_{\nu}} b_j.
        \end{align*}
		We obtain \eqref{eq: parametric regularity of partial x psi} for $|\nu|$ by iterating the above equation from $n$ to $2$.
	\end{proof}
	
	\begin{lemma} \label{lem: parametric regularity of functional}
		Under Assumptions \ref{assp: general}--\ref{assp: summability of b}, we have for any $T > 0$ and $\bxi \in U_a$
		\begin{align} \label{eq: parametric regularity of functional}
			\vert \partial^{\nu} G \left( |\psi_m^n(\bxi)|^2 \right) \vert \leq C_{\mathrm{a}, 1} (|\nu| + 1)! 4^{|\nu| + 2} C_{T}^{|\nu| + 2}  \| G \|_{H^1(\bT)'} \| \psi_{\mathrm{in}} \|_{H^1(\bT)}^2 (\Upsilon_m(\bxi))^2 \prod_{j \in \cI_{\nu}} b_j,
		\end{align}
		where $1 \le n \le \lfloor T / \tau \rfloor$.
	\end{lemma}
	
	\begin{proof}
		Fix $T > 0$ and $\bxi \in U_a$. For $1 \le n \le \lfloor T / \tau \rfloor$, we have
		\begin{align*}
			\vert \partial^{\nu} G \left( |\psi_m^n|^2 \right) \vert = \vert G \left( \partial^{\nu} |\psi_m^n|^2 \right) \vert \le \Vert G \Vert_{H^1(\bT)'} \| \partial^{\nu} |\psi_m^n|^2 \|_{H^1(\bT)}.
		\end{align*}
		By the algebraic property of $H^1(\bT)$, we find
		\begin{align*}
			\vert \partial^{\nu} G \left( |\psi_m^n(\bxi)|^2 \right)  \vert \leq & \Vert G \Vert_{H^1(\bT)'} \sum_{{\mu} \le {\nu}} C_{\mathrm{a}, 1}
			\left\|\partial^{{\nu}-{\mu}} \overline{\psi_m^n}\right\|_{H^1(\bT)} \|\partial^{\mu}\psi_m^n\|_{H^1(\bT)}.
		\end{align*}
		We obtain \eqref{eq: parametric regularity of functional} from the above equation by Lemma \ref{lem: parametric regularity of semidiscrete in H1}.
	\end{proof}
	
	Lemma \ref{lem: parametric regularity of functional} shows that $\vert \partial^{\nu} G \left( |\psi_m^n(\bxi)|^2 \right) \vert$ grows at most quadratically in $|\bxi_m|$. For the analysis of QMC quadrature error, we work with the non-standard weighted unanchored Sobolev space $\mathcal{W}_{m,\gamma, w}:=\{F(\xi): \bR^m \to \bR\  |\ \|F\|_{\cW_{m,\gamma, w}}<\infty\}$ \cite{nichols2014fast} with the norm
	\begin{align} \label{eq: definition of norm of weighted space}
		\|F\|_{\cW_{m,\gamma, w}} := \left[ \sum_{\nu \in \cJ} \gamma_\nu^{-1}
		\int_{\bR^{|\nu|}} \left(\int_{\bR^{m-|\nu|}} \partial^\nu F(\xi) \prod_{j \in \cI_{\nu}^{\mathsf{c}}} \phi(\xi_j) \rd \xi_{\nu}^{\mathsf{c}}\right)^2 \prod_{j \in \cI_{\nu}} w_j(\xi_j)^{2} \rd\xi_{\nu} \right]^{1/2},
	\end{align}
	where by the notation, $\xi$ is split into the active part $\xi_{\nu}$ for differentiation and the inactive part $\xi_{\nu}^{\mathsf{c}}$, i.e., $\xi_{\nu}$ consists of $\xi_j$ with $j \in \cI_{\nu}$ and $\xi_{\nu}^{\mathsf{c}}$ consists of $\xi_k$ with $k \in \cI_{\nu}^{\mathsf{c}}$. The notation $\gamma_\nu$ denotes the product and order-dependent (POD) weight parameter, i.e., $\gamma_\nu = \Gamma_{|\nu|} \prod_{j \in \cI_{\nu}}\gamma_j$ with $\gamma_{(0, \ldots, 0)}=1$ for some sequences $\Gamma_0 = \Gamma_1 = 1, \Gamma_2, \ldots$ and $\gamma_1 \ge \gamma_2 \ge \ldots > 0$, which will be chosen in Lemma \ref{lem: QMC quadrature error}, and $\gamma:=\{\gamma_{\nu}\,|\, \nu\in\cJ\}$. The function $w_j$ in \eqref{eq: definition of norm of weighted space} is a positive continuous decaying weight function that serves to counteract the growth of $| \partial^{\nu} G ( |\psi_m^n(\bxi)|^2| ) |$, and $w := \{w_j: j = 1, \ldots, m\}$.
	The quadratic growth of the bound in \eqref{eq: parametric regularity of functional} allows us to choose the favorable weight functions in the form of
	\begin{align} \label{eq: weight function}
		w_j(\xi) = \exp \left( - \theta_j |\xi| \right),
	\end{align}
	which can lead to a dimension-independent and almost linear convergence rate of QMC \cite{kuo2010randomly,nichols2014fast}; see Lemmas \ref{lem: worst error bound}--\ref{lem: QMC quadrature error}.
	Here, we assume that $\max\{ b_j, \theta_{\min} \} < \theta_j < \theta_{\max}$ for $j \in \bN^{+}$ and some constants $0 < \theta_{\min} < \theta_{\max} < \infty$, and we will specify $\theta_j$ in Lemma \ref{lem: QMC quadrature error}. We also let $\cD = \inf_{j \in \bN}(\theta_j - b_j)$.
	
	Recall the randomly shifted QMC lattice rule \eqref{eq: QMC for F} with the generating vector $\bz$ and the random shift $\bDelta$, and we have from \cite{graham2015quasi,nichols2014fast} that for a general $F \in \cW_{m, \gamma, w}$
	\begin{align} \label{eq: worst case error}
		\sqrt{\mathbb{E}^{\bDelta}\left[\Big|\E[F] - Q_{m,N}(F; \bDelta)\Big|^2 \right]} \le e_{m, N}^{\mathrm{sh}}(\bz) \| F \|_{\cW_{m, \gamma, w}},
	\end{align}
	where $e_{m, N}^{\mathrm{sh}}(\bz)$ is referred to as the shift averaged worst case error. Note that $\| F \|_{\cW_{m, \gamma, w}}$ does not depend on $\bz$, and thus we can construct $\bz$ by making $e_{m, N}^{\mathrm{sh}}(\bz)$ as small as possible. Moreover, we know from \cite{nichols2014fast} that $e_{m, N}^{\mathrm{sh}}(\bz)$ depends on the weight functions $w$ and the POD weight parameters $\gamma$, and that once $w$ and $\gamma$ are chosen a generating vector $\bz$ can be constructed by the CBC algorithm. That is, we set $z_1 = 1$ and then determine $z_j$ for $j = 2, \ldots, m$ sequentially by minimizing $e_{j, N}^{\mathrm{sh}}(z_1, \ldots, z_{j - 1}, z_j)$ over $z_j \in \{ z \in \bN: 1 \le z \le N - 1, \mathrm{gcd}(z, N) = 1 \}$. We refer to \cite{nichols2014fast} for the explicit expression of $e_{m, N}^{\mathrm{sh}}(\bz)$ and more details on the CBC algorithm. More importantly, it is proved in \cite{nichols2014fast} that we can obtain almost linear convergence in $N$ for $e_{m, N}^{\mathrm{sh}}(\bz)$ with the standard multivariate Gaussian distribution and our choice of weight functions in \eqref{eq: weight function}. We present a relevant result from \cite[Theorem 15]{graham2015quasi} in Lemma \ref{lem: worst error bound}. Based on it, we prove the dimension-independent and almost linear convergence rate in $N$ of the QMC-TS scheme in Lemma \ref{lem: QMC quadrature error}.
	
	\begin{lemma}[{\cite[Theorem 15]{graham2015quasi}}]
		\label{lem: worst error bound}
		Let $m, N \in \bN^{+}$, the weight parameters $\gamma$ be fixed, the weight functions $w$ be given in the form of \eqref{eq: weight function}, and $F \in \cW_{m,\gamma, w}$. Then, there exists a randomly shifted lattice rule \eqref{eq: QMC for F} that can be constructed by the CBC algorithm such that
		\begin{align} \label{eq: worst error bound}
			\sqrt{\mathbb{E}^{\bDelta} \left[ |\E[F] - Q_{m,N}(F; \bDelta)|^2 \right]}
			\leq \left[ \sum_{\nu \in \cJ \setminus \{ (0, \ldots, 0) \}}\gamma_\nu^\lambda
			\prod_{j \in \cI_{\nu}} \varrho_j(\lambda)
			\right]^{\frac{1}{2\lambda}}
			\varphi_{\mathrm{tot}}(N)^{-\frac{1}{2\lambda}}\|F\|_{W_{m,\gamma, w}},
		\end{align}
		for any $\lambda \in (\frac{1}{2},1]$, with
		\begin{align} \label{eq: expression of varrho}
			\varrho_j(\lambda) = 2 \left( \frac{ \sqrt{2\pi} \exp( \theta_j^2 / \eta_* ) }{
				\pi^{2 - 2 \eta_*} (1 - \eta_*) \eta_* } \right)^{\lambda} \zeta \left( \lambda + \frac{1}{2} \right), \quad \eta_* = \frac{2 \lambda - 1}{4 \lambda}
		\end{align}
		where $\zeta$ is the Riemann zeta function and $\varphi_{\mathrm{tot}}$ is the Euler totient function with the property that $1/\varphi_{\mathrm{tot}}(N) \le 9 / N$ for $N \le 10^{30}$.
	\end{lemma}
	
	\begin{lemma} \label{lem: QMC quadrature error}
		Let Assumptions \ref{assp: general}--\ref{assp: summability of b} hold. If Assumption \ref{assp: summability of b} holds with $p = 1$, we additionally assume that
		\begin{align} \label{eq: additional assumption}
			\sum_{j = 1}^{\infty} b_j < \frac{\cD^{1/2}}{4 C_T \varrho_{\max}(1)^{1/2}},
		\end{align}
		where $\varrho_{\max}(\lambda)$ is defined by replacing $\theta_j$ in \eqref{eq: expression of varrho} by $\theta_{\max}$ and recall that $\cD = \inf_{j \in \bN}(\theta_j - b_j)$. Choose the weight parameters
		\begin{align} \label{eq: choice of weight parameter}
			\gamma_{\nu} = \left( C_{|\nu|} \prod_{j \in \cI_{\nu}} \frac{b_j^2}{(\theta_j - b_j)\varrho_j(\lambda^*)} \right)^{1 / (1 + \lambda^*)},
		\end{align}
		where $C_{|\nu|} = ((|\nu| + 1)!)^2 (4 C_T)^{2|\nu|}$,
		\begin{equation} \label{eq: choice of convergence rate}
			\lambda^* = \left\{
			\begin{aligned}
				& \frac{1}{2 - 2 \delta}, & & \text{ when } 0 < p \le \frac{2}{3}, \\
				& \frac{p}{2 - p}, & & \text{ when } \frac{2}{3} < p \le 1,
			\end{aligned}
			\right.
		\end{equation}
		with arbitrary $\delta \in (0, 1/2]$, and $\theta_j$ is the parameter in the weight function $w_j(\xi_j)$ in \eqref{eq: weight function} with
		\begin{align} \label{eq: choice of parameters in weight functions}
			\theta_j = \frac{1}{2} \left( b_j + \sqrt{b_j^2 + 1 - \frac{1}{2 \lambda^*}} \right), \quad j \in \bN^+.
		\end{align}
		Then, there exists a randomly shifted lattice rule \eqref{eq: QMC for F} that can be constructed by the CBC algorithm such that for $N \le 10^{30}$
		\begin{equation} \label{eq: QMC quadrature error}
			\sqrt{\mathbb{E}^{\bDelta} \left[ |\E[G(|\psi_m^n|^2)] - Q_{m,N}(G(|\psi_m^n|^2); \bDelta)|^2 \right]} \le
			\left\{
			\begin{aligned}
				& C N^{-(1 - \delta)}, & & \text{ when } 0 < p \le \frac{2}{3}, \\
				& C N^{-(1/p - 1/2)}, & & \text{ when } \frac{2}{3} < p \le 1,
			\end{aligned}
			\right.
		\end{equation}
		where $C$ is independent of $m, \tau, N$, but depends on $p$ and, when relevant, $\delta$.
	\end{lemma}
	
	\begin{proof}
		With the elementary inequality $\max\{ 1, \rho^2 \} \le \exp(\rho)$ for $\rho \ge 0$, we can deduce from Lemma \ref{lem: parametric regularity of functional} that
		\begin{align*}
			\vert \partial^{\nu} G \left( |\psi_m^n(\bxi)|^2 \right) \vert \le \widetilde{C} (|\nu| + 1)! (4 C_T)^{|\nu|} \exp\left( \sum_{j = 1}^m b_j |\xi_j| \right) \prod_{j \in \cI_{\nu}} b_j, \quad 1 \le n \le \lfloor T / \tau \rfloor, \bxi \in U_a,
		\end{align*}
		where $\widetilde{C} = 16 C_{\mathrm{a}, 1} C_{T}^2 \| G \|_{H^1(\bT)'} \| \psi_{\mathrm{in}} \|_{H^1(\bT)}^2 \exp(\| v_0 \|_{W^{1, \infty}(\bT)})$. Then, using $\int_{\bR} \exp(2 b_j |\xi_j|) w_j^2(\xi_j) \rd \xi_j = 1 / (\theta_j - b_j)$ from the proof of \cite[Theorem 16]{graham2015quasi}, we obtain by the definition \eqref{eq: definition of norm of weighted space} of the $\cW_{m, \gamma, w}$-norm that
        \begin{align*}
            \| G \left( |\psi_m^n|^2 \right) \|_{\cW_{m, \gamma, w}}^2 \le & \widetilde{C}^2 \prod_{j=1}^m ( 2 \exp(b_j^2 / 2) \Phi(b_j) ) \left( \sum_{\nu \in \cJ} \gamma_\nu^{-1} C_{|\nu|} \prod_{j \in \cI_{\nu}} \frac{b_j^2}{2 \exp(b_j^2 / 2) \Phi(b_j) (\theta_j - b_j)} \right).
        \end{align*}
        Still from the proof of \cite[Theorem 16]{graham2015quasi}, we have $2 \exp( b_j^2 / 2) \Phi(b_j) \ge 1$ and $2 \exp(b_j^2 / 2) \Phi(b_j) \le \exp(b_j^2 / 2 + 2 b_j / \sqrt{2 \pi})$; recall also \eqref{eq: property of density function}. Then, we can deduce from the above equation that
        \begin{align}
			\| G \left( |\psi_m^n|^2 \right) \|_{\cW_{m, \gamma, w}}^2 \le (\breve{C})^2 \sum_{\nu \in \cJ} \gamma_\nu^{-1} C_{|\nu|} \prod_{j \in \cI_{\nu}} \frac{b_j^2}{\theta_j - b_j}, \label{eq: Sobolev norm of functional}
		\end{align}
		where $C_{|\nu|} = ((|\nu| + 1)!)^2 (4 C_T)^{2|\nu|}$ and $\breve{C} = \widetilde{C} \exp \left( \frac{1}{4} \sum_{j = 1}^{\infty} b_j^2 + \frac{1}{\sqrt{2 \pi}} \sum_{j = 1}^{\infty} b_j \right)$.
		
		The rest of the proof is similar to that of \cite[Theorem 20 and Corollary 21]{graham2015quasi}, and we only provide the main idea here. In view of Lemma \ref{lem: worst error bound} and \eqref{eq: Sobolev norm of functional}, to derive a dimension-independent bound on the QMC quadrature error, we need to bound the quantity
		\begin{align*}
			C_{\gamma, m} := \left( \sum_{\nu \in \cJ \setminus \{(0, \ldots, 0)\}}\gamma_\nu^\lambda
			\prod_{j \in \cI_{\nu}} \varrho_j(\lambda)
			\right)^{1/\lambda} \left( \sum_{\nu \in \cJ} \gamma_\nu^{-1} C_{|\nu|} \prod_{j \in \cI_{\nu}} \frac{b_j^2}{\theta_j - b_j} \right)
		\end{align*}
		independently of $m$. To this end, we first choose
		\begin{align*}
			\gamma_{\nu} = \left( C_{|\nu|} \prod_{j \in \cI_{\nu}} \frac{b_j^2}{(\theta_j - b_j)\varrho_j(\lambda)} \right)^{1 / (1 + \lambda)},
		\end{align*}
		which minimizes $C_{\gamma, m}$ for fixed $\theta_j$ and $\lambda$ by \cite[Lemma 6.2 and Theorem 6.4]{kuo2012quasi}. Then, we further find that $\lambda \in (1/2, 1]$ needs to be bounded from below, depending on the value of $p$ in Assumption \ref{assp: summability of b}. For $\lambda \in (1/2, 1)$, we need $\lambda \ge p / (2 - p)$. We want $\lambda$ to be as small as possible in view of \eqref{eq: worst error bound}. So we can choose $\lambda = \lambda^* = 1 / (2 - 2 \delta)$ for some $\delta \in (0, 1/2)$ when $p \in (0, 2/3]$ and $\lambda = \lambda^* = p / (2 - p)$ when $p \in (2/3, 1)$. When $p = 1$, we need to additionally assume \eqref{eq: additional assumption}, and then we can choose $\lambda = \lambda^* = 1$. Finally, the choice of $\{ \theta_j \}_{j = 1}^{\infty}$ in \eqref{eq: choice of parameters in weight functions} minimizes $C_{\gamma, m}$ given the above choices of $\gamma_{\nu}$ and $\lambda$. These choices of $\gamma_{\nu}, \lambda$ and $\{ \theta_j \}_{j = 1}^{\infty}$ give the error estimate \eqref{eq: QMC quadrature error} by Lemma \ref{lem: worst error bound}.
	\end{proof}
	
	Now we are ready to prove the main result presented in Section \ref{subsec: main result}.
	
	\begin{proof}[Proof of Theorem \ref{thm: main}]
		By the triangle inequality, we have
		\begin{align*}
			& \sqrt{\mathbb{E}^{\bDelta}\left[\Big|\E[G(|\psi(t_n)|^2)] - Q_{m,N}(G(|\psi_{m}^n|^2); \bDelta)\Big|^2 \right]} \\
			& \qquad \le | \E[G(|\psi(t_n)|^2)] - \E[G(|\psi_m(t_n)|^2)] | + | \E[G(|\psi_m(t_n)|^2)] - \E[G(|\psi_m^n|^2)] | \\
			& \qquad \quad + \sqrt{\mathbb{E}^{\bDelta} \left[ |\E[G(|\psi_m^n|^2)] - Q_{m,N}(G(|\psi_m^n|^2); \bDelta)|^2 \right]}.
		\end{align*}
		Then, we prove the theorem using Lemmas \ref{lem: dimension truncation of physical observable}, \ref{lem: temporal error of functional}, and \ref{lem: QMC quadrature error}.
	\end{proof}
	
	\subsection{Some discussions}
	\label{subsec: discussion}
	
	We end this section with a few comments on our main result.
	
	\begin{enumerate}
		\item The time discretization scheme has a direct impact on the performance of the randomly shifted lattice-based QMC quadrature rule. If the semi-discrete solution grows too fast in $|\bxi|$, we may have to choose weight functions that decay faster than \eqref{eq: weight function} and the resulting QMC quadrature may not achieve almost linear convergence \cite{kuo2010randomly,nichols2014fast}.
		
		\item Theorem \ref{thm: main} can be generalized to the $d$-space-dimensional case ($d \ge 1$) by additionally assuming that $G \in (H^r(\bT^d))^{\prime}$, $v_j \in W^{r, \infty}(\bT^d)$ for $j \in \bN^{+}$, $\sum_{j = 1}^{\infty} (b_{j, r})^p < \infty$ for some $p \in (0, 1]$ with $b_{j, r} = \lambda_j \|
		v_j \|_{W^{r, \infty}(\bT^d)}$, and $\| V_m - V \|_{L_{\bmu_{\infty}}^{2 + \varepsilon}(U, H^r(\bT^d))}\leq Cm^{-\chi}$ for some constants $C, \chi, \varepsilon$ independent of $m$, where $r > d / 2$. Then, the error estimate \eqref{eq: main estimate} still holds for QMC-TS.
		
		\item We can use a general splitting scheme of the form \eqref{eq: high-order splitting} for time discretization. In particular, we can adopt the Strang splitting \cite{strang1968construction}
		\begin{align} \label{eq: Strang splitting}
			\psi_m^{n + 1} = \Psi^\mathrm{k}_{\tau/2} \circ \Psi^\mathrm{p}_{\tau} \circ \Psi^\mathrm{k}_{\tau/2} (\psi_m^{n}) = \re^{\ri \tau \partial_x^2 / 4} \re^{- \ri \tau V_m} \re^{\ri \tau \partial_x^2 / 4} \psi_m^n, \quad n = 0, 1, \ldots,
		\end{align}
		which is known to be second-order and is one of the most popular splitting schemes. In this case, if we additionally assume that Assumptions \ref{assp: general}--\ref{assp: convergence of V} hold with $s \ge 5$, then under the conditions of Theorem \ref{thm: main} we can prove by following the analysis in \cite[Section 4]{su2020time} that the following error estimate holds for QMC-TS:
		\begin{align}
			\sqrt{\mathbb{E}^{\bDelta}\left[\Big|\E[G(|\psi(t_n)|^2)] - Q_{m,N}(G(|\psi_{m}^n|^2); \bDelta)\Big|^2
				\right]} \leq C
			(m^{-\chi} + \tau^{2} + N^{- \kappa}),
		\end{align}
		where $C$ is independent of $m, \tau, N$ and $\kappa$ is the same as in Theorem \ref{thm: main}.
		
		\item Following the third comment, for the Lie--Trotter splitting \eqref{eq: Lie splitting} (resp. the Strang splitting \eqref{eq: Strang splitting}), the time convergence of order $1$ (resp. order $2$) will be achieved as long as Assumption \ref{assp: general} holds with $s \ge 3$ (resp. $s \ge 5$), while the summability of $\{ a_j \}_{j = 1}^{\infty}$ with $s \ge 3$ (resp. $s \ge 5$) in Assumption \ref{assp: summability of b} guarantees that the time convergence is independent of $m$.
		
		\item Following the second and third comments, the QMC quadrature achieves the dimension-independent $\mathcal{O}(N^{- \kappa})$ convergence as long as $ \sum_{j = 1}^{\infty} (b_{j, r})^p < \infty$ for some $p \in (0, 1]$, regardless of the choice of the splitting scheme for time discretization and whether the summability of $\{ a_j \}_{j = 1}^{\infty}$ is satisfied.
		
		\item The fully discrete scheme can be obtained by combining the QMC-TS scheme with some spatial discretization, e.g.,  the Fourier pseudospectral method \cite{shen2011spectral}, the finite difference method \cite{markowich1999numerical,markowich2002wigner}, etc. The crucial step to obtain the QMC convergence rate for the fully discrete scheme is to derive a bound on the mixed first derivatives of the fully discrete solution with respect to $\bxi$. However, the spatial discretization would introduce a constant in the estimate of the fully discrete solution at each time step. This constant is typically larger than $1$ and thus would lead to blow up of the estimate as $n \rightarrow \infty$. But we believe that the QMC convergence rate in Theorem \ref{thm: main} would still hold provided that the spatial mesh size is sufficiently small; see the numerical examples in Section \ref{sec: numerical example}. We will address this issue in future works.

        \item Although we mainly consider a linear functional of the solution as the quantity of interest in this paper, we could possibly extend our results to Banach-space-valued quantity of interest as in \cite{guth2024parabolic}, where uniformly distributed random variables are considered. Let $G: U \rightarrow Z$ be a continuous Banach-space-valued mapping, where $Z$ is a separable Banach space. For our problem, $G$ could be $G(\bxi_m) = \psi_m^n(\bxi_m, \cdot), G(\bxi_m) =  |\psi_m(\bxi_m, \cdot)|^2$, etc. Using Hahn-Banach theorem \cite[Theorem 4.3]{rudin1991funcional}, the worst case error \eqref{eq: worst case error} can be extended to the Banach-space-valued mapping $G$; see \cite[Theorem 6.5]{guth2024parabolic}. Then, by deriving the parametric regularity bound on $\| \partial^{\nu} G \|_Z$ and combining the worst case error for Banach-space-valued mappings with the error bound of the CBC algorithm \eqref{eq: worst error bound}, we could derive the error bound for QMC applied to the Banach-space-valued mapping $G$; see \cite[Theorem 6.6]{guth2024parabolic}. We refer the reader to \cite[Section 6]{guth2024parabolic} for more details.

        \item We can also consider higher-order QMC methods developed in, e.g., \cite{dick2014higher,dick2016multilevel}, for our problem, and the convergence analysis of higher-order methods would depend on parametric regularity analysis for higher-order mixed partial derivatives. We would leave this for future study.

        \item A more challenging problem is the nonlinear Schr\"{o}dinger equation with a Gaussian random potential
        \begin{equation}\label{eq: NLS}
			\ri \partial_t\psi(t,\omega,x) = -\frac{1}{2} \partial_x^2 \psi(t,\omega,x) + V(\omega,x) \psi(t,\omega,x) + \alpha |\psi(t,\omega,x)|^2 \psi(t,\omega,x),
        \end{equation}
        where $V(\omega,x)$ is a Gaussian random field and $\alpha \neq 0$. The well-posedness of the solution $\psi$ is non-trivial due to the cubic nonlinearity. To be concrete, following the proof of Lemma \ref{lem: well-posedness of psi}, we need to prove that $\exp(C_{\mathrm{a}, 1} T \| \psi(t, \bxi) \|_{H^1(\bT)}^2)$ is $L^q$-integrable with respect to $\bxi$ in order to derive the integrability of $\psi$ in the random space. However, $\psi$ is typically non-Gaussian with respect to $\bxi$, and thus Fernique's theorem cannot be applied to prove this integrability. We will study this problem in the future.
	\end{enumerate}
	
	\section{Numerical examples}
	\label{sec: numerical example}
	
	In this section, we give some convergence tests of the QMC-TS method to verify our theoretical findings. To implement the QMC-TS method, we use the Fourier pseudospectral method with mesh size $h$ for spatial discretization, and the generating vector for QMC is constructed using the code script ``lat-cbc.py'' from \cite{QMC4PDE} (see also \cite{kuo2016application}). We consider two primary physical observables: the position density $S$ and the current density $J$, where
	\begin{align}
		S(t, \bxi, x) = & S(\psi(t, \bxi, x)) = | \psi(t, \bxi, x) |^2, \\
		J(t, \bxi, x) = & J(\psi(t, \bxi, x)) = \mathrm{Im}(\overline{\psi(t, \bxi, x)} \nabla \psi(t, \bxi, x)).
	\end{align}

    To measure the temporal error and dimension truncation error, we consider the $L^2$ relative errors for $S$ and $J$, i.e.,
	\begin{align}
		\mathrm{err}_{L^2}(S) = & \frac{\| \E_{\mathrm{num}}[S(\psi_{m, \mathrm{num}}(t))] - \E_{\mathrm{ref}}[S(\psi_{m, \mathrm{ref}}(t))] \|_{L^2(\bT)}}{\| \E_{\mathrm{ref}}[S(\psi_{m, \mathrm{ref}}(t))] \|_{L^2(\bT)}}, \label{eq: L2 error of S} \\
		\mathrm{err}_{L^2}(J) = & \frac{\| \E_{\mathrm{num}}[J(\psi_{m, \mathrm{num}}(t))] - \E_{\mathrm{ref}}[J(\psi_{m, \mathrm{ref}}(t))] \|_{L^2(\bT)}}{\| \E_{\mathrm{ref}}[J(\psi_{m, \mathrm{ref}}(t))] \|_{L^2(\bT)}}. \label{eq: L2 error of J}
	\end{align}
	Here, $\E_{\mathrm{ref}}[S(\psi_{m, \mathrm{ref}}(t, x))]$ and $\E_{\mathrm{ref}}[J(\psi_{m, \mathrm{ref}}(t, x))]$ are the reference solutions. Moreover, $\psi_{m, \mathrm{num}}$ is the numerical solution obtained by some numerical scheme in time and space for fixed $m$ and $\bxi_m$. For QMC as the sampling method, $\E_{\mathrm{num}}[F]$ for a general $F(\bxi_m)$ reads
	\begin{align*}
		\E_{\mathrm{num}}[F] = \overline{Q}_{m, N, R}(F) = \frac{1}{R} \sum_{k = 1}^R Q_{m, N}(F; \bDelta_k),
	\end{align*}
	where $Q_{m, N}(F; \bDelta_k)$ is defined in \eqref{eq: QMC for F} with $\bDelta_k$ the $k$-th independent random shift.

    On the other hand, to measure the sampling error, we consider the standard error, which is an unbiased estimator of the root mean square error (RMSE). The same idea can be found in \cite{gilbert2019analysis,graham2015quasi,wu2024error}. More precisely, for a general $F(\bxi_m)$, the standard error for QMC reads
	\begin{align} \label{eq: error proxy}
		\sqrt{ \frac{1}{R(R - 1)} \sum_{k = 1}^R \left( Q_{m, N}(F; \bDelta_k) - \overline{Q}_{m, N, R}(F) \right)^2 } \approx \sqrt{\E^{\bDelta}[ | \E[ F ] - Q_{m, N}(F; \bDelta) |^2 ]},
	\end{align}
	where the left-hand side is the standard error of $F$ and the right-hand side is RMSE of $F$ for QMC. We will also consider the sampling error of MC. For MC, the standard error reads \cite{dick2013high}
    \begin{align}
        \sqrt{ \frac{1}{N_{\mathrm{MC}} (N_{\mathrm{MC}} - 1)} \sum_{j = 1}^{N_{\mathrm{MC}}} ( F(\bxi_{m, \mathrm{MC}}^{(j)}) - Q^{\mathrm{MC}}_{m, N_{\mathrm{MC}}}(F) )^2 } \approx \sqrt{ \E[ | \E[F] - Q^{\mathrm{MC}}_{m, N_{\mathrm{MC}}}(F) |^2 ] },
    \end{align}
    where the left-hand side is the standard error of $F$ and the right-hand side is RMSE of $F$ for MC. Here, $\{ \bxi_{m, \mathrm{MC}}^{(j)} \}_{j = 1}^{N_{\mathrm{MC}}}$ are the MC sample points with $N_{\mathrm{MC}}$ the total number, and $Q^{\mathrm{MC}}_{m, N_{\mathrm{MC}}}(F) = \frac{1}{N_{\mathrm{MC}}} \sum_{j = 1}^{N_{\mathrm{MC}}} F(\bxi_{m, \mathrm{MC}}^{(j)})$.
    Let $S_T(x) = S(\psi_{m, \mathrm{num}}(t = T, \bxi_m, x))$ and $J_T(x) = J(\psi_{m, \mathrm{num}}(t = T, \bxi_m, x))$. We would consider the standard errors of $S_T(\frac{\pi}{4})$ and $J_T(\frac{\pi}{4})$ to measure the sampling error in all the following examples.

	The first example tests the convergence rates of different time-splitting schemes and sampling methods.
	
	\begin{example} \label{expl: comparison of convergence rates}
		Consider the Schr\"{o}dinger equation \eqref{eq: Schrodinger trun} with $\bT = [-\pi, \pi]$, $T = 1$, the initial data
		\begin{align} \label{eq: numerical initial condition}
			\psi_{\mathrm{in}}(x) = \sqrt{\frac{8}{\pi}} \exp(- 8 x^2),
		\end{align}
		and the random potential
		\begin{align} \label{eq: numerical potential 1}
			V_m(\bxi, x) = 1 + \sum_{j = 1}^m \frac{1}{j^{\frac{9}{2}}} \xi_j \cos(j x).
		\end{align}
		We consider $m = 4$.
	\end{example}
	For time convergence tests, the reference solutions are computed using the Strang splitting and the Fourier pseudospectral method combined with the stochastic collocation method, where we choose $\tau = 5 \times 10^{-5}, h = \frac{\pi}{128}$ and $20$ collocation points in each of the $m$ dimensions of $\bxi_m$. For numerical solutions, we use the Lie--Trotter and Strang splittings in QMC-TS, where we fix $R = 50, N = 2^{18}, h = \frac{\pi}{64}$ and choose $\tau = \frac{1}{40}, \frac{1}{80}, \frac{1}{160}, \frac{1}{320}, \frac{1}{640}$. By Point (4) in Section \ref{subsec: discussion}, we should observe first-order and second-order convergence in $\tau$ for the Lie--Trotter and Strang splittings, respectively. The numerical results are shown in Figure \ref{fig: smallD_time}, where the optimal convergence rates in time are observed for both splitting schemes.

	For tests of convergence in the number of samples, we compare QMC and MC as the sampling method, and we use the Strang splitting and the Fourier pseudospectral method for time and spatial discretization, respectively. We fix $\tau = 10^{-4}$ and $h = \frac{\pi}{64}$. For a fair comparison of QMC and MC, we let $N_{\mathrm{tot}} = N_{\mathrm{MC}} = R N$, and we choose $R = 50$ and $N = 2^{10}, 2^{11}, \ldots, 2^{16}$. The standard errors of $S_{T}(\frac{\pi}{4})$ and $J_{T}(\frac{\pi}{4})$ are shown in Figure \ref{fig: smallD_random}, which is plotted against $N_{\mathrm{tot}}$ for convenience. The convergence rates of the standard errors for QMC and MC, which are fitted using errors corresponding to $N = 2^{13}, 2^{14}, 2^{15}, 2^{16}$, are shown in Table \ref{tab: smallD}. We observe that the convergence rates of QMC are approximately linear, which is consistent with Point(5) in Section \ref{subsec: discussion}, while the convergence rates of MC are near the theoretical value $\frac{1}{2}$.
	
	\begin{figure}[t!]
		\centering
		\subcaptionbox{$L^2$ relative error of $S$ at $T = 1$}{\includegraphics[width=0.45\textwidth]{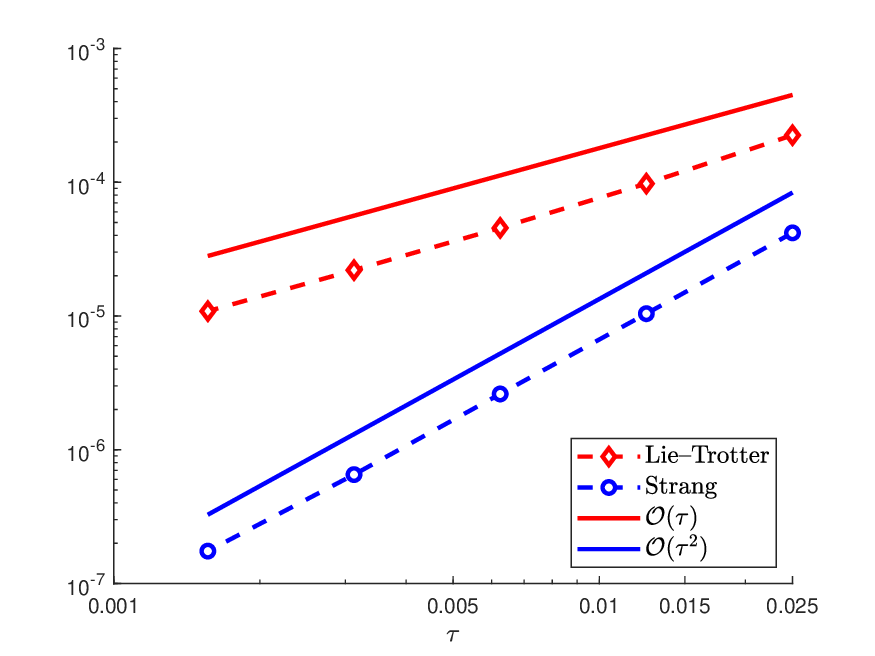}}
		\subcaptionbox{$L^2$ relative error of $J$ at $T = 1$}{\includegraphics[width=0.45\textwidth]{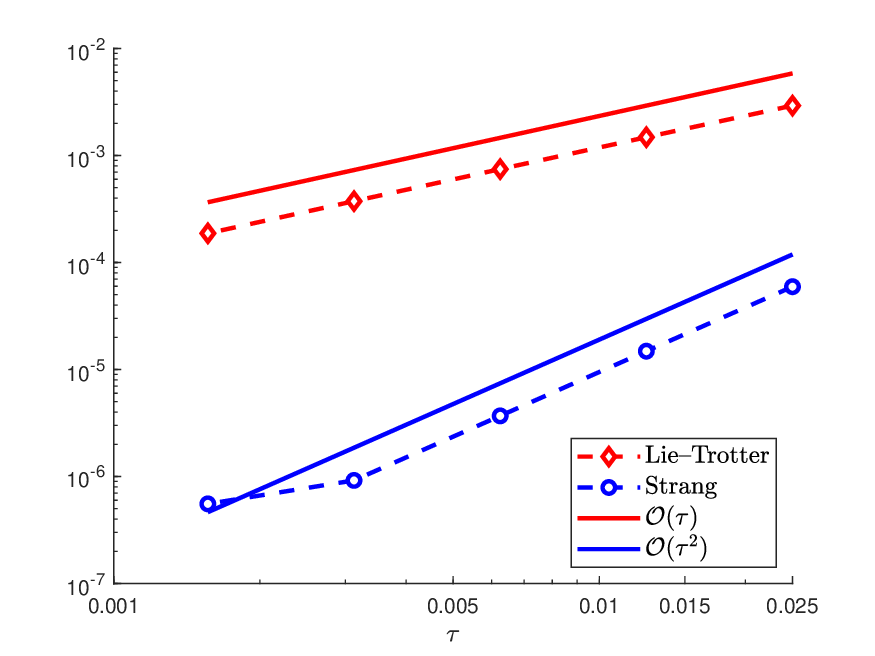}}
		\caption{Convergence in time for Example \ref{expl: comparison of convergence rates}\label{fig: smallD_time}}
	\end{figure}
	
	\begin{figure}[t!]
		\centering
		\subcaptionbox{Standard error of $S_{T}(\frac{\pi}{4})$ at $T = 1$}{\includegraphics[width=0.45\textwidth]{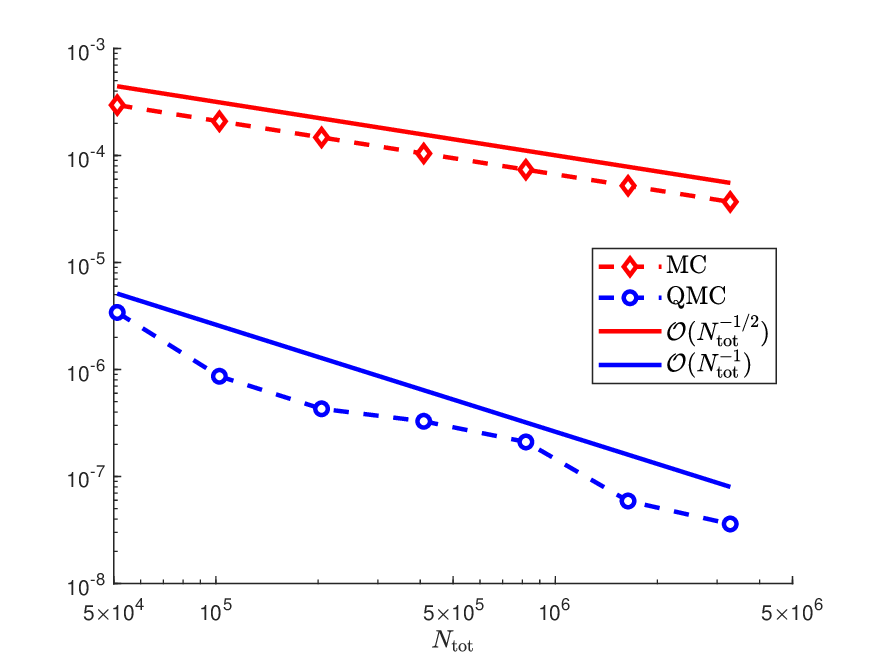}}
		\subcaptionbox{Standard error of $J_{T}(\frac{\pi}{4})$ at $T = 1$}{\includegraphics[width=0.45\textwidth]{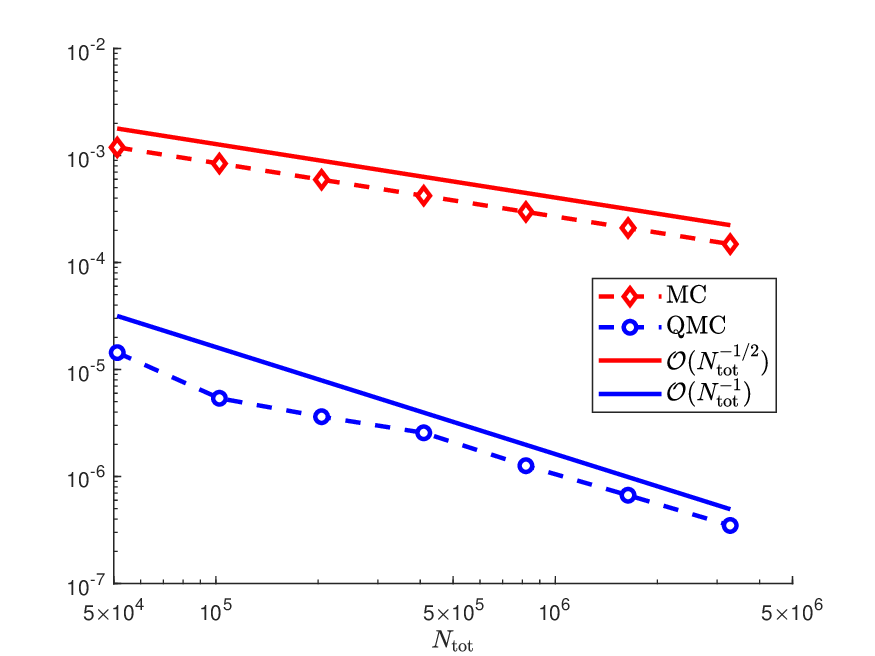}}
		\caption{Convergence in the number of samples for Example \ref{expl: comparison of convergence rates}\label{fig: smallD_random}}
	\end{figure}

	\begin{table}[t!]
		\begin{subtable}{0.45\linewidth}
			\centering
			\begin{tabular}{|c|c|c|}
				\hline
				& QMC & MC \\
				\hline
				$S_T(\frac{\pi}{4})$ & 1.1398 & 0.5004 \\
				$J_T(\frac{\pi}{4})$ & 0.9548 & 0.5007 \\
				\hline
			\end{tabular}
			\caption{Example \ref{expl: comparison of convergence rates}}
			\label{tab: smallD}
		\end{subtable}
		\begin{subtable}{0.45\linewidth}
			\centering
			\begin{tabular}{|c|c|c|c|}
				\hline
				& $m = 2$ & $m = 4$ & $m = 6$ \\
				\hline
				$S_T(\frac{\pi}{4})$ & 0.9338 & 1.1398 & 1.1107 \\
				$J_T(\frac{\pi}{4})$ & 1.0737 & 0.9548 & 1.0256 \\
				\hline
			\end{tabular}
			\caption{Example \ref{expl: dimension independence in small D}}
			\label{tab: smallmultiD_QMC}
		\end{subtable}
		\caption{Fitted convergence rates of standard errors for Examples \ref{expl: comparison of convergence rates}--\ref{expl: dimension independence in small D}}
	\end{table}
	
	The second example shows the dimension-independence of the convergence in time and the number of QMC samples.
	\begin{example} \label{expl: dimension independence in small D}
		Consider the Schr\"{o}dinger equation \eqref{eq: Schrodinger trun} with $\bT = [-\pi, \pi], T = 1$, the initial condition \eqref{eq: numerical initial condition} and the random potential \eqref{eq: numerical potential 1}. We consider $m = 2, 4, 6$.
	\end{example}
	For time convergence tests, the reference solutions are computed in the same way as in the previous example, and we use the Lie--Trotter splitting in QMC-TS for numerical solutions. We should observe dimension-independent and first-order convergence in time by Theorem \ref{thm: main}. We fix $R = 50, N = 2^{18}, h = \frac{\pi}{64}$ and choose $\tau = \frac{1}{40}, \frac{1}{80}, \frac{1}{160}, \frac{1}{320}, \frac{1}{640}$. The results are shown in Figure \ref{fig: smallmultiD_Lie}. The $L^2$ relative errors decay at a rate of $\mathcal{O}(\tau)$ and are nearly the same for different $m$, which is consistent with Theorem \ref{thm: main}.

	For convergence tests of QMC, we use the Strang splitting in QMC-TS. We should observe dimension-independent and almost linear convergence in the number of samples by Point (5) in Section \ref{subsec: discussion}. We fix $\tau = 10^{-4}, h = \frac{\pi}{64}$, and choose $R = 50$ and $N = 2^{10}, 2^{11}, \ldots, 2^{16}$. The standard errors of $S_{T}(\frac{\pi}{4})$ and $J_{T}(\frac{\pi}{4})$ are shown in Figure \ref{fig: smallmultiD_QMC}. The convergence rates of standard errors, which are fitted using errors corresponding to $N = 2^{13}, 2^{14}, 2^{15}, 2^{16}$, are shown in Table \ref{tab: smallmultiD_QMC}. The dimension-independence of the QMC quadrature error is confirmed by Figure \ref{fig: smallmultiD_QMC}, and Table \ref{tab: smallmultiD_QMC} shows that the convergence rates are all around $1$ for different $m$, which verifies our theories.
	
	\begin{figure}[t!]
		\centering
		\subcaptionbox{$L^2$ relative error of $S$ at $T = 1$}{\includegraphics[width=0.45\textwidth]{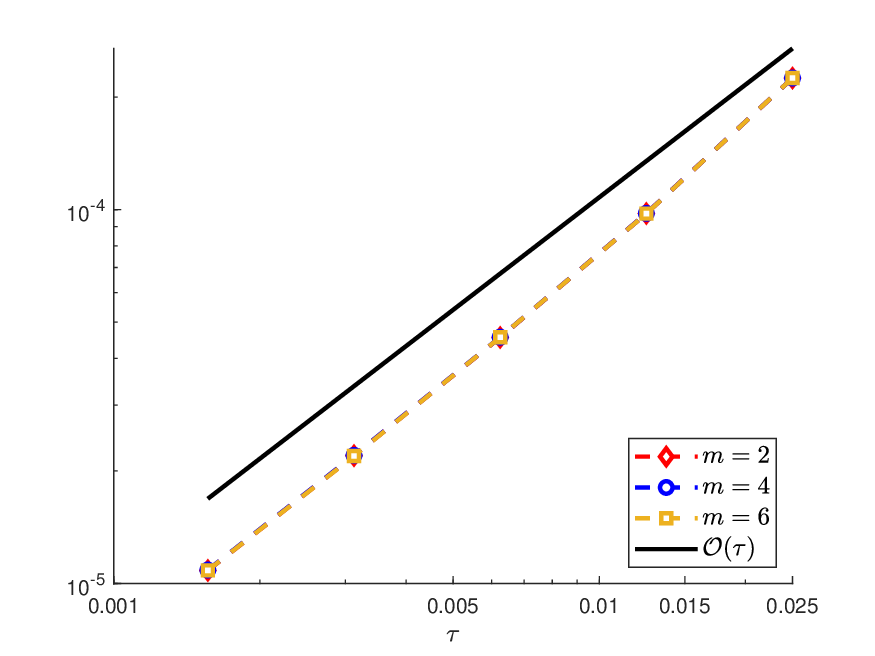}}
		\subcaptionbox{$L^2$ relative error of $J$ at $T = 1$}{\includegraphics[width=0.45\textwidth]{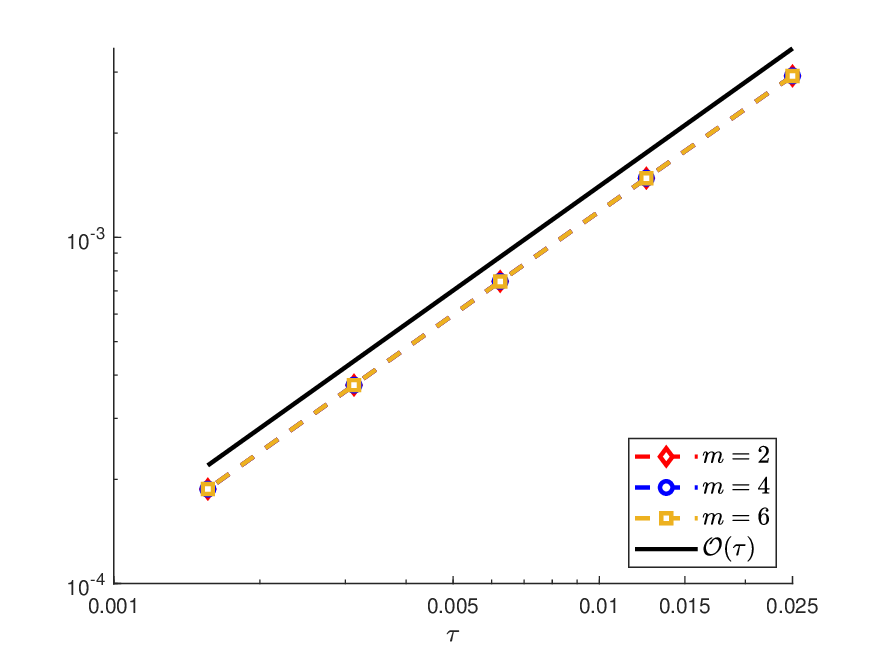}}
		\caption{Convergence in time for Example \ref{expl: dimension independence in small D}\label{fig: smallmultiD_Lie}}
	\end{figure}
	
	\begin{figure}[t!]
		\centering
		\subcaptionbox{Standard error of $S_{T}(\frac{\pi}{4})$ at $T = 1$}{\includegraphics[width=0.45\textwidth]{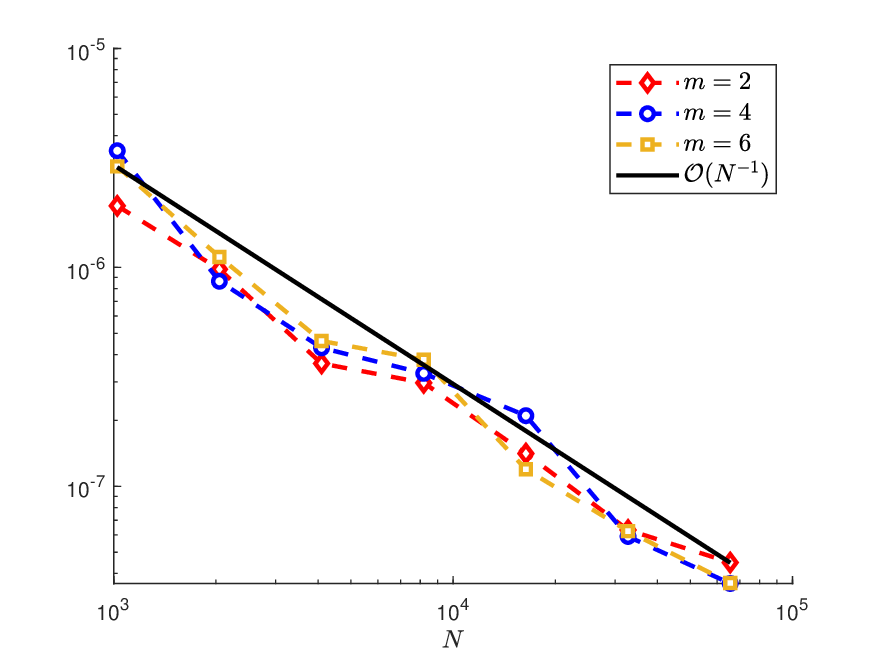}}
		\subcaptionbox{Standard error of $J_{T}(\frac{\pi}{4})$ at $T = 1$}{\includegraphics[width=0.45\textwidth]{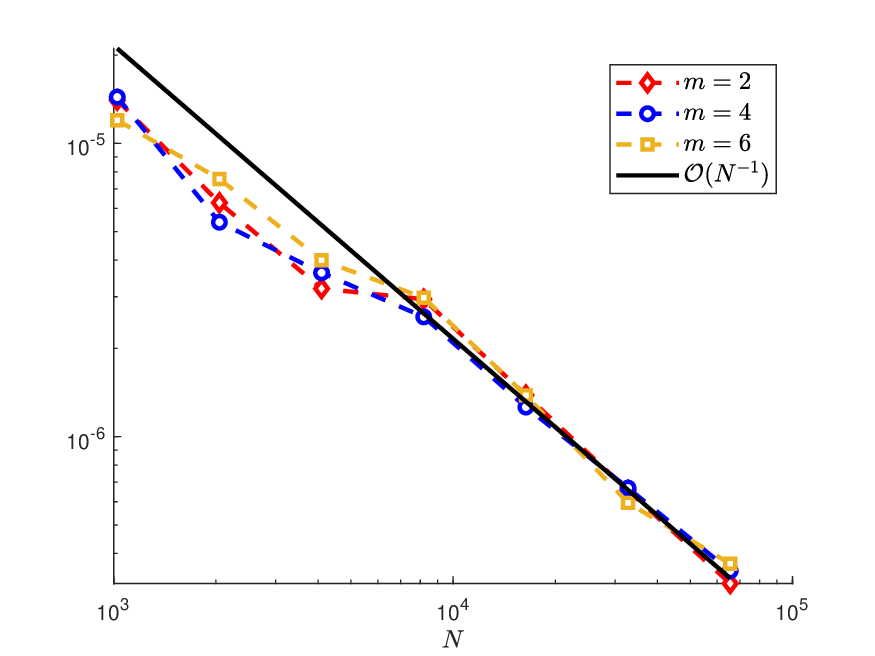}}
		\caption{Convergence in the number of samples for Example \ref{expl: dimension independence in small D}\label{fig: smallmultiD_QMC}}
	\end{figure}
	
	The next example shows the dimension-independence of the convergence of QMC in high dimensions.
	\begin{example} \label{expl: dimension independence in large D}
		Consider the Schr\"{o}dinger equation \eqref{eq: Schrodinger trun} with $\bT = [-\pi, \pi], T = 1$, the initial condition \eqref{eq: numerical initial condition} and the random potential \eqref{eq: numerical potential 1}. We consider $m = 8, 12, 16$.
	\end{example}
	
	We use Lie--Trotter splitting in QMC-TS. We should observe dimension-independent and almost linear convergence of QMC by Theorem \ref{thm: main}. We fix $\tau = 2 \times 10^{-5}, h = \frac{\pi}{64}$ and choose $R = 50$ and $N = 2^{10}, 2^{11}, \ldots, 2^{16}$. The standard errors of $S_T(\frac{\pi}{4})$ and $J_T(\frac{\pi}{4})$ are shown in Figure \ref{fig: largeD}. The convergence rates of the standard errors, which are fitted using errors corresponding to $N = 2^{13}, 2^{14}, 2^{15}, 2^{16}$, are shown in Table \ref{tab: largeD}. The standard errors do not vary much as $m$ changes, which confirms the dimension-independence of QMC quadrature error, and the convergence rates are close to $1$. These results validate our theories.
	
	\begin{figure}[t!]
		\centering
		\subcaptionbox{Standard error of $S_T(\frac{\pi}{4})$ at $T = 1$}{\includegraphics[width=0.45\textwidth]{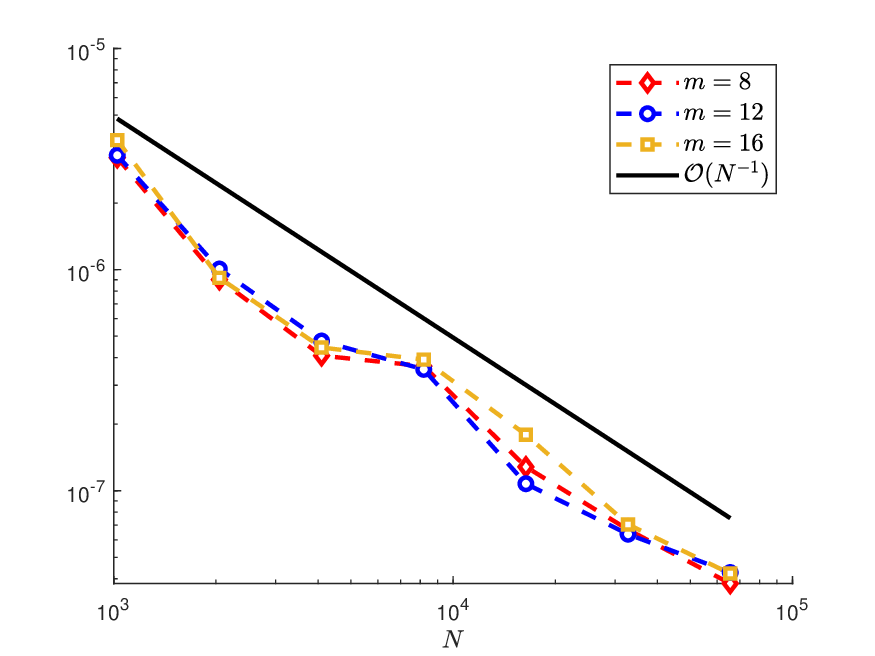}}
		\subcaptionbox{Standard error of $J_T(\frac{\pi}{4})$ at $T = 1$}{\includegraphics[width=0.45\textwidth]{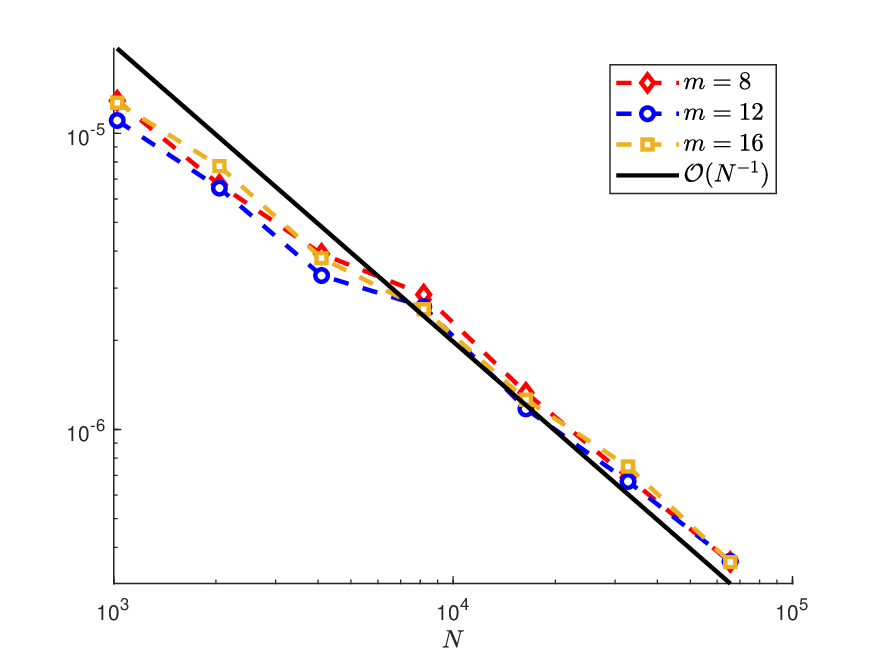}}
		\caption{Convergence of QMC in Example \ref{expl: dimension independence in large D}\label{fig: largeD}}
	\end{figure}
	
	\begin{table}[t!]
		\begin{subtable}{\linewidth}
			\centering
			\begin{tabular}{|c|c|c|c|}
				\hline
				& $m = 8$ & $m = 12$ & $m = 16$ \\
				\hline
				$S_T(\frac{\pi}{4})$ & 1.0738 & 0.9912 & 1.0989 \\
				$J_T(\frac{\pi}{4})$ & 0.9951 & 0.9379 & 0.9253 \\
				\hline
			\end{tabular}
			\caption{Example \ref{expl: dimension independence in large D}}
			\label{tab: largeD}
		\end{subtable}
		\newline
		\newline
		\begin{subtable}{0.45\linewidth}
			\begin{tabular}{|c|c|c|c|}
				\hline
				& $m = 8$ & $m = 12$ & $m = 16$ \\
				\hline
				$S_T(\frac{\pi}{4})$ & 0.7732 & 0.7850 & 0.7578 \\
				$J_T(\frac{\pi}{4})$ & 0.7890 & 0.8074 & 0.8737 \\
				\hline
			\end{tabular}
			\caption{Example \ref{expl: linear and sublinear convergence of QMC} with $\alpha = 9 / 4$}
			\label{tab: sublinear convergence}
		\end{subtable}
		\begin{subtable}{0.45\linewidth}
			\begin{tabular}{|c|c|c|c|}
				\hline
				& $m = 8$ & $m = 12$ & $m = 16$ \\
				\hline
				$S_T(\frac{\pi}{4})$ & 1.1881 & 1.1551 & 1.1459 \\
				$J_T(\frac{\pi}{4})$ & 0.9119 & 0.9201 & 0.9381 \\
				\hline
			\end{tabular}
			\caption{Example \ref{expl: linear and sublinear convergence of QMC} with $\alpha = 5 / 2$}
			\label{tab: linear convergence}
		\end{subtable}
		\caption{Fitted convergence rates of standard errors for Examples \ref{expl: dimension independence in large D}--\ref{expl: linear and sublinear convergence of QMC}}
	\end{table}
	
	The next example shows the effect of the decay rate of the randomness on the convergence rate of QMC.
	\begin{example} \label{expl: linear and sublinear convergence of QMC}
		Consider the Schr\"{o}dinger equation \eqref{eq: Schrodinger trun} with $\bT = [-\pi, \pi], T = 1$, the initial condition \eqref{eq: numerical initial condition} and the random potential
		\begin{align} \label{eq: numerical potential 2}
			V_m(\bxi, x) = 1 + \sum_{j = 1}^m \frac{1}{j^{\alpha}} \xi_j \cos(j x).
		\end{align}
		We consider $\alpha = \frac{9}{4}, \frac{5}{2}$ and $m = 8, 12, 16$.
	\end{example}

    Note that letting $\alpha = \frac{9}{4}, \frac{5}{2}$ would not satisfy $\sum_{j = 1}^{\infty} a_j < \infty$ with $s \ge 3$ in Assumption \ref{assp: summability of b}. However, we can still apply our QMC-TS to this example with finite $m$ by Point (4) in Section \ref{subsec: discussion} (although the temporal error may not be independent of $m$). Moreover, also note that $\sum_{j = 1}^{\infty} b_j^p < \infty$ is satisfied for some $p \in (0, 1]$ for both $\alpha$. By Point (5) in Section \ref{subsec: discussion}, we would obtain dimension-independent QMC convergence regardless of the splitting scheme for time discretization and whether $\sum_{j = 1}^{\infty} a_j < \infty$ is satisfied. Furthermore, Theorem \ref{thm: main} shows that the convergence rate of QMC would be approximately $\frac{3}{4}$ if $\alpha = \frac{9}{4}$ and would be almost linear if $\alpha = \frac{5}{2}$. Recall that we consider the standard errors of $S_T(\frac{\pi}{4})$ and $J_T(\frac{\pi}{4})$. Note that the RMSE (and hence the standard error) is independent of the temporal error, and thus the standard error can correctly show the QMC convergence even if the temporal error depends on $m$.
	
	We use the Strang splitting in QMC-TS. We fix $\tau = 10^{-4}, h = \frac{\pi}{64}$, and choose $R = 50$ and $N = 2^{10}, 2^{11}, \ldots, 2^{15}$. The standard errors are shown in Figure \ref{fig: linear and sublinear}, and the fitted convergence rates are shown in Tables \ref{tab: sublinear convergence} and \ref{tab: linear convergence}. We still see from Figure \ref{fig: linear and sublinear} that the standard errors do not vary much with different values of $m$, which again confirms the dimension-independence of convergence of QMC. Furthermore, we see from Table \ref{tab: sublinear convergence} that for $\alpha = \frac{9}{4}$ the fitted convergence rate of the standard error of $S_T(\frac{\pi}{4})$ is around the theoretical value $\frac{3}{4}$ and that of $J_T(\frac{\pi}{4})$ is slightly larger than $\frac{3}{4}$. On the other hand, we see from Table \ref{tab: linear convergence} for $\alpha = \frac{5}{2}$, the fitted convergence rates of the standard errors of both $S_T(\frac{\pi}{4})$ and $J_T(\frac{\pi}{4})$ are around the theoretical value $1$. These results show that faster decay of randomness in the potential leads to faster convergence of QMC, which is consistent with Theorem \ref{thm: main}.
	
	\begin{figure}[t!]
		\centering
		\subcaptionbox{Standard error of $S_T(\frac{\pi}{4})$ with $\alpha = \frac{9}{4}$}{\includegraphics[width=0.45\textwidth]{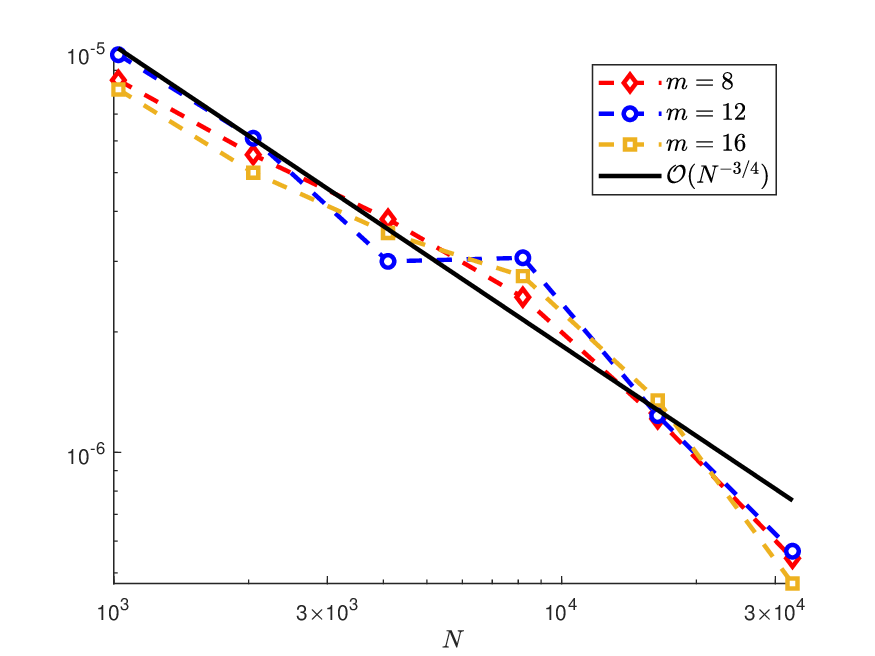}}
		\subcaptionbox{Standard error of $J_T(\frac{\pi}{4})$ with $\alpha = \frac{9}{4}$}{\includegraphics[width=0.45\textwidth]{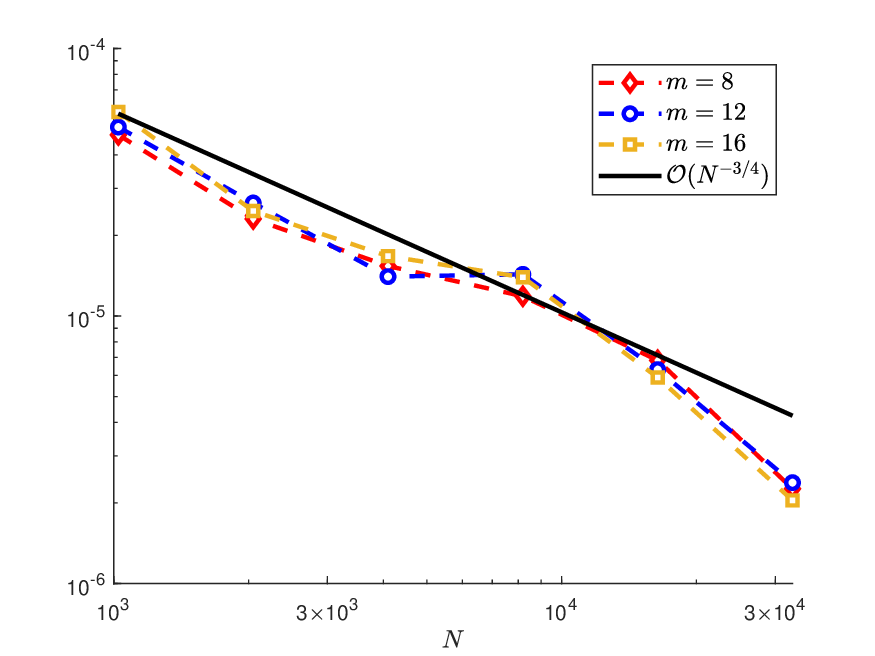}}
		\subcaptionbox{Standard error of $S_T(\frac{\pi}{4})$ with $\alpha = \frac{5}{2}$}{\includegraphics[width=0.45\textwidth]{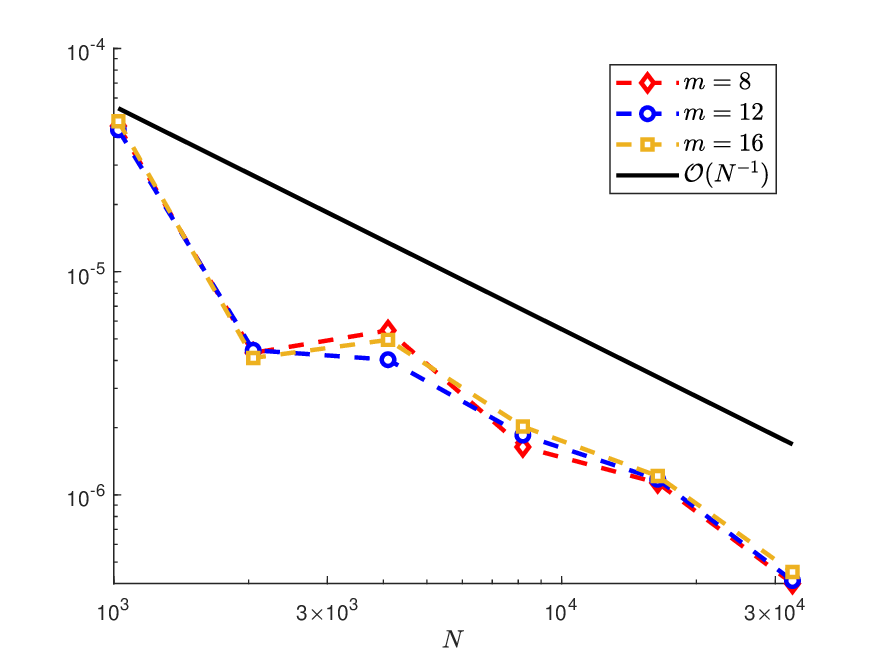}}
		\subcaptionbox{Standard error of $J_T(\frac{\pi}{4})$ with $\alpha = \frac{5}{2}$}{\includegraphics[width=0.45\textwidth]{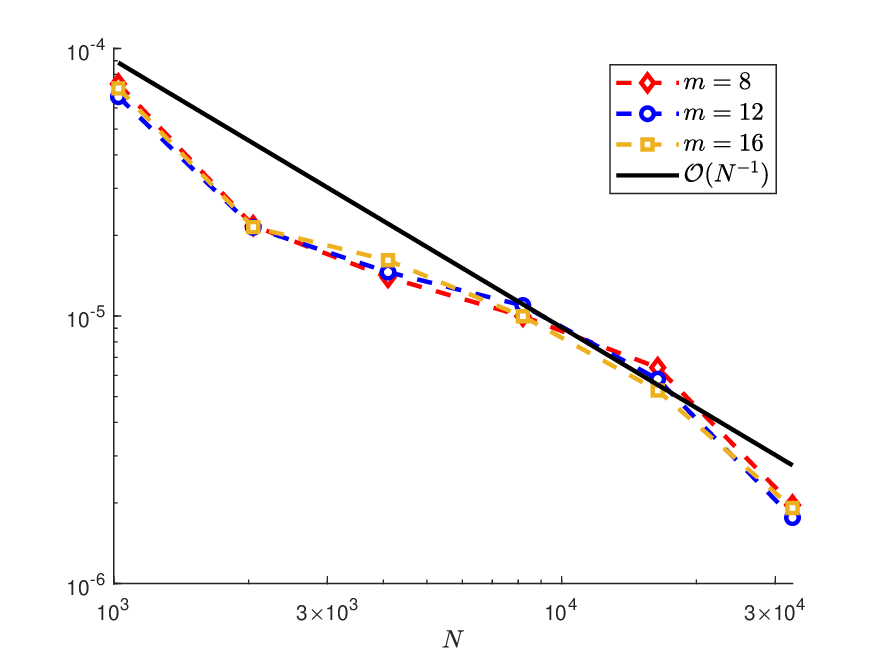}}
		\caption{Convergence of QMC in Example \ref{expl: linear and sublinear convergence of QMC}, with $T = 1$\label{fig: linear and sublinear}}
	\end{figure}

    We show the convergence of dimension truncation and explore the performance of our method in an extremely high-dimensional case in the last example.

    \begin{example} \label{expl: extreme highD}
        Consider the Schr\"{o}dinger equation \eqref{eq: Schrodinger trun} with $\bT = [-\pi, \pi], T = 1$, the initial condition \eqref{eq: numerical initial condition} and the random potential
		\begin{align} \label{eq: numerical potential 3}
			V_m(\bxi, x) = 1 + \sum_{j = 1}^m \frac{1}{j^{\alpha}} \xi_j v_j(x), \text{ with }
            v_j(x) = \left\{
            \begin{aligned}
                & \cos(kx), \quad j = 2k - 1, \\
                & \sin(kx), \quad j = 2k.
            \end{aligned}
            \right.
		\end{align}
    \end{example}

    We first consider the convergence of dimension truncation. We let $\alpha = \frac{5}{2}, 3$. Note that these two values of $\alpha$ do not make $\sum_{j = 1}^{\infty} a_j < \infty$ hold, but $\sum_{j = 1}^{\infty} b_j < \infty$ is satisfied. Moreover, we can see from the proofs in Section \ref{subsec: dimension truncation error} that the dimension truncation error estimate in Lemma \ref{lem: dimension truncation of physical observable} can be obtained as long as $\sum_{j = 1}^{\infty} b_j < \infty$. Therefore, it is feasible to test the convergence of dimension truncation using this example with the chosen values of $\alpha$. By Theorem \ref{thm: main}, the convergence rates $\chi$ of dimension truncation should be $\frac{1}{2}$ and $1$ for $\alpha = \frac{5}{2}, 3$, respectively. For each $\alpha$, we compute the reference solutions using Strang splitting in QMC-TS with $m = 128, \tau = 10^{-4}, h = \frac{\pi}{128}, R = 50$, and $N = 2^{15}$. The numerical solutions are also computed using Strang splitting in QMC-TS with the same $\tau, h, R, N$, where we choose $m = 2, 4, 8, 16, 32, 64$. The $L^2$ relative errors are shown in Figure \ref{fig: extreme_highD_dimension}, where we see that the convergence rates of dimension truncation are much faster than the theoretical values, which indicates that our error estimate for dimension truncation in Theorem \ref{thm: main} might not be optimal.

    Finally, we test QMC convergence of our method in an extremely high-dimensional case. We let $\alpha = \frac{21}{10}$ and $m = 100$. Using the same arguments as in the previous example, it is feasible to test the QMC convergence in this example using this value of $\alpha$ (although we do not have $\sum_{j = 1}^{\infty} a_j < \infty$ in this case). Furthermore, by Theorem \ref{thm: main}, the QMC convergence rate would be approximately $\frac{3}{5}$ for $\alpha = \frac{21}{10}$. The numerical solutions are computed using Strang splitting in QMC-TS with $\tau = 10^{-4}, h = \frac{\pi}{128}$, where we choose $R = 50$ and $N = 2^{10}, 2^{11}, \ldots, 2^{16}$. The standard errors of $S_{T}(\frac{\pi}{4})$ and $J_{T}(\frac{\pi}{4})$ are shown in Figure \ref{fig: extreme_highD_QMC}, and the convergence rates, which are fitted using errors corresponding to $N = 2^{12}, 2^{13}, \ldots, 2^{16}$, are 0.6706 and 0.6725 for the standard errors of $S_{T}(\frac{\pi}{4})$ and $J_{T}(\frac{\pi}{4})$, respectively, which are consistent with Theorem \ref{thm: main}. Therefore, our QMC-TS works well for the sampling in our problem in extremely high-dimensional cases with slow decay of randomness, and the QMC convergence rate is consistent with the theory we develop.

    \begin{figure}[t!]
		\centering
		\subcaptionbox{$L^2$ relative error of $S$ at $T = 1$}{\includegraphics[width=0.45\textwidth]{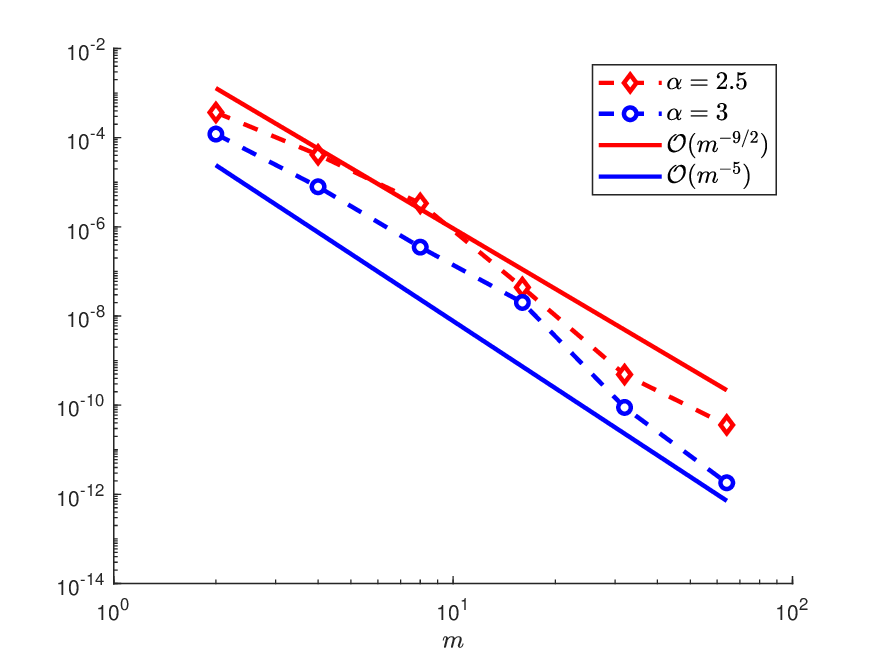}}
		\subcaptionbox{$L^2$ relative error of $J$ at $T = 1$}{\includegraphics[width=0.45\textwidth]{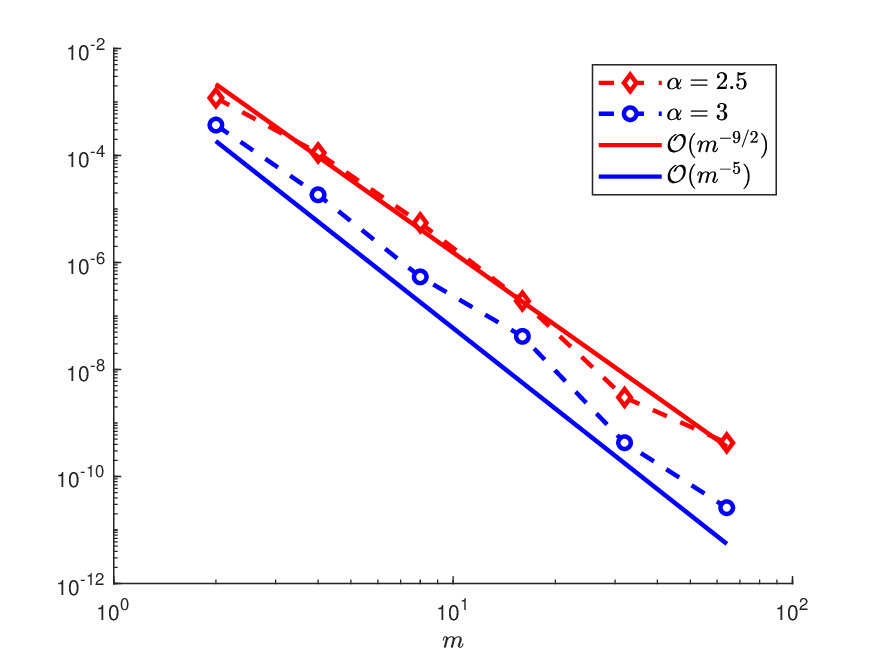}}
		\caption{Convergence of dimension truncation for Example \ref{expl: extreme highD}\label{fig: extreme_highD_dimension}}
	\end{figure}

    \begin{figure}[t!]
		\centering
		\subcaptionbox{Standard error of $S_{T}(\frac{\pi}{4})$}{\includegraphics[width=0.45\textwidth]{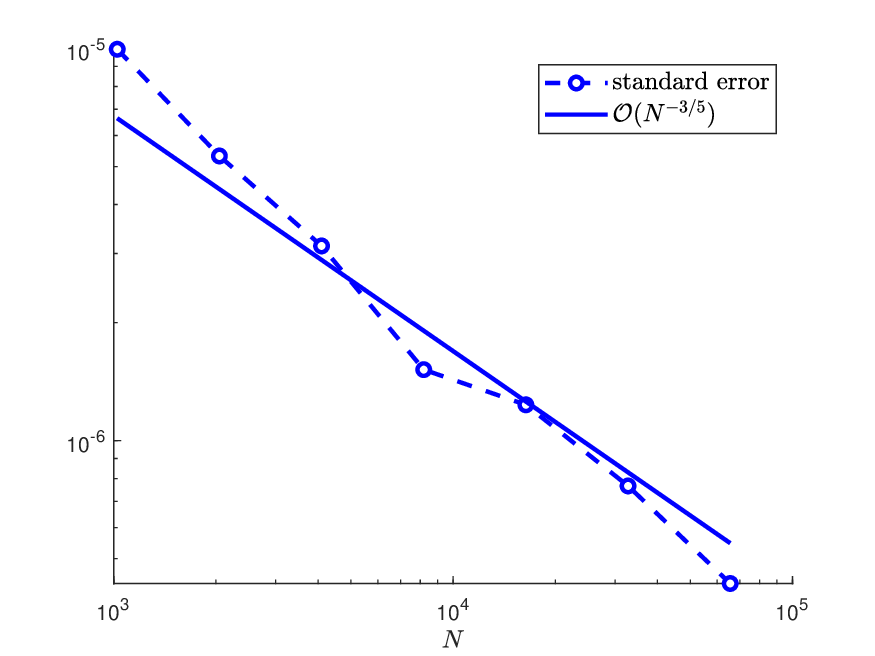}}
		\subcaptionbox{Standard error of $J_{T}(\frac{\pi}{4})$}{\includegraphics[width=0.45\textwidth]{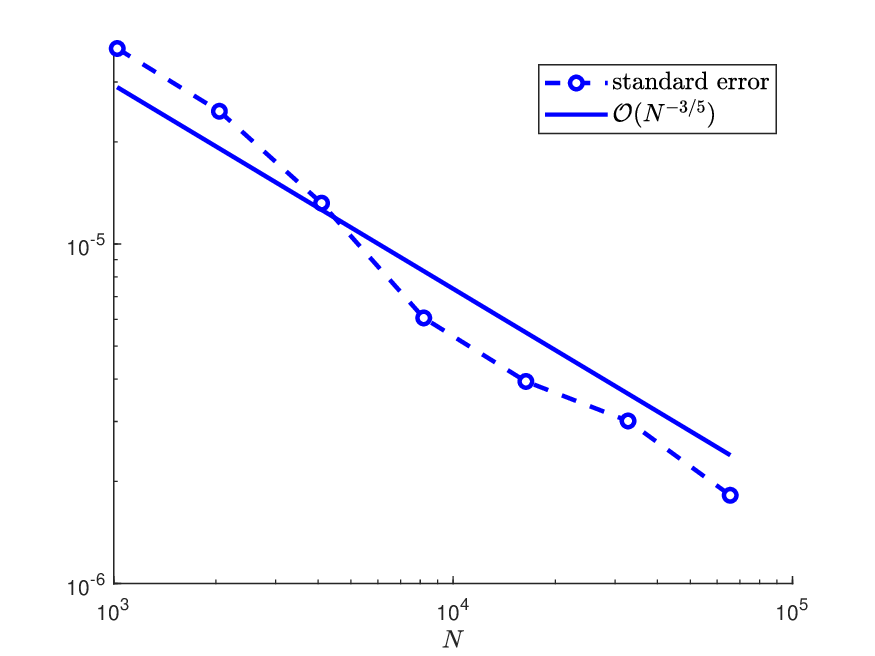}}
		\caption{Convergence of QMC for Example \ref{expl: extreme highD}, with $T = 1$ and $\alpha = \frac{21}{10}$\label{fig: extreme_highD_QMC}}
	\end{figure}

	\section{Conclusion}
	\label{sec:Conclusion}
	
	We developed a quasi-Monte Carlo time-splitting scheme to solve the Schr\"{o}dinger equation with a Gaussian random potential, with a particular focus on approximating the expectation of physical observables. By the technique of a non-standard weighted Sobolev space, we proved a dimension-independent convergence rate of the proposed scheme, which is almost linear with respect to the number of QMC samples. The sharpness of our theoretical error estimates was demonstrated by numerical experiments, which shows the efficiency of our scheme. In the future, we will extend the current work by considering the fully discrete scheme and the nonlinear Schr\"{o}dinger equation.

	\section*{Acknowledgement}
	X. Zhao is supported by the National Key Research and Development Program of China (Project 2024YFE03240400) and the National Natural Science Foundation of China (Projects 42450275 and 12271413). Z. Zhang was supported by the National Natural Science Foundation of China  (Projects 92470103 and 12171406), the Hong Kong RGC grant (Projects 17304324 and 17300325), seed funding from the HKU-TCL Joint Research Center for Artificial Intelligence, and the Outstanding Young Researcher Award of HKU (2020-21). The authors would like to thank Professor Ivan G. Graham (University of Bath), Professor Frances Kuo (University of New South Wales), and Doctor James A. Nichols (Australian National University) for helpful discussions on the coding for constructing the generating vector by the CBC algorithm. The computations were performed using research computing facilities provided by Information Technology Services, The University of Hong Kong.
	
	\bibliographystyle{siam}
	\bibliography{reference}
	
\end{document}